\documentclass[11pt,a4paper,reqno]{amsart}
\usepackage{enumerate}
\usepackage[T1]{fontenc}
\usepackage{exscale,color,epsfig}
\usepackage{amsmath,amsthm,amssymb,amsfonts,latexsym}
\usepackage[active]{srcltx}
\usepackage{latexsym}
\usepackage{pstricks,pst-grad}
\usepackage{tikz}  
\usepackage{float}

\newcommand{\Tr}{{\mathop{\mathrm{Tr} \,}}}

\newcommand{\per}{\mathrm{per}}

\DeclareMathOperator{\supp}{supp}
\DeclareMathOperator{\dist}{dist}
\DeclareMathOperator{\diam}{diam}
\newcommand{\LL}{{\Lambda_L}}
\renewcommand{\L}{\Lambda}
\newcommand{\DeloneSet}{Z}
\newcommand{\tL}{Q}

\newcommand{\norm}[1]{\Vert #1 \Vert}
\newcommand{\ceil}[1]{\lceil #1 \rceil}

\newcommand{\e}{\mathrm e}

\newcommand{\ra}{\rangle}
\newcommand{\la}{\langle}

     \newcommand{\EE}{\mathbb{E}}
     
     \newcommand{\NN}{\mathbb{N}}
     \newcommand{\PP}{\mathbb{P}}
     
     \newcommand{\RR}{\mathbb{R}}
     \newcommand{\ZZ}{\mathbb{Z}}

     \newcommand{\cD}{\mathcal{D}}

     \newcommand{\cI}{\mathcal{I}}

\newcommand{\abs}[1]{\left| #1 \right|}
\renewcommand{\L}{\Lambda}

\newtheorem{thm}{Theorem}[section]
\newtheorem{lem}[thm]{Lemma}

\newtheorem{cor}[thm]{Corollary}
\theoremstyle{definition}
\newtheorem{dfn}[thm]{Definition}
\newtheorem{hyp}[thm]{Assumption}
\theoremstyle{remark}
\newtheorem{rem}[thm]{Remark}
\newtheorem{exm}[thm]{Example}
\newcommand{\be}{\begin{equation}}
\newcommand{\ee}{\end{equation}}
\newcommand{\bea}{\begin{eqnarray*}}
\newcommand{\eea}{\end{eqnarray*}}
\newcommand{\beq}{\begin{eqnarray}}
\newcommand{\eeq}{\end{eqnarray}}
\newcommand{\nn}{\nonumber}

\newcommand{\hm}[1]{\leavevmode{\marginpar{\tiny%
$\hbox to 0mm{\hspace*{-0.5mm}$\leftarrow$\hss}%
\vcenter{\vrule depth 0.1mm height 0.1mm width \the\marginparwidth}%
\hbox to
0mm{\hss$\rightarrow$\hspace*{-0.5mm}}$\\\relax\raggedright #1}}}

\begin{document}

\title[Scale-free unique continuation principle]
{Scale-free unique continuation estimates and applications to random Schr\"odinger operators}

\author[C.~Rojas-Molina]{Constanza Rojas-Molina}
\address[C.~R-M.]{Universit\'e de Cergy-Pontoise, UMR CNRS 8088, F-95000 Cergy-Pontoise, France}
\email{constanza.rojas-molina@u-cergy.fr}
\author[I.~Veseli\'c]{Ivan Veseli\'c}
\address[I.~V.]{ Fakult\"at f\"ur Mathematik,\,  TU-Chemnitz, Germany  }
\urladdr{www.tu-chemnitz.de/mathematik/stochastik}

\thanks{{\today, \jobname.tex}\\
This file consist of two parts. The first (including Appendix A) is the final manuscript submitted in September '12 to CMP.
Appendix B is part of the first version of February '12. 
It concerns an estimate on local mass fluctuations inside dominating boxes, a result which is not used in the main body of the paper, but may be of independent interest.}


\maketitle

\begin{abstract}
We prove a unique continuation principle or uncertainty relation
valid for Schr\"odinger operator eigenfunctions, or more generally solutions of a
Schr\"odinger  inequality, on cubes of side $L\in 2\NN+1$.
It establishes an equi-distribution property of the eigenfunction over the box:
the total $L^2$-mass in the box of side $L$ is estimated from above by a constant times the sum of the $L^2$-masses on
small balls of a fixed radius $\delta>0$ evenly distributed throughout the box.
The dependence of the constant on the various parameters entering the problem
is given  explicitly. Most importantly, there is no $L$-dependence.

This result has important consequences for the perturbation theory of eigenvalues of Schr\"odinger operators, in particular random ones. For so-called Delone-Anderson models we deduce Wegner estimates, a lower bound for the shift of the spectral minimum, and an uncertainty relation for spectral projectors.
\end{abstract}

\section{Introduction}

The results in the present paper concern two areas of mathematics: partial differential equations and the spectral  theory of Schr\"odinger operators,
in particular random ones.
The theorems related to the first topic are rather general in nature and may have several applications.
We use them to derive spectral estimates for Schr\"odinger operators with random potentials of alloy type, which do not need to exhibit
a stochastic translational invariance.

Let us first informally describe our main result about the behaviour of eigen-solutions of a Schr\"odinger operator
$H_L= -\Delta+V$ on a cube $\L_L$ of side $L\in 2\NN+1$, with Dirichlet or periodic boundary conditions.
We assume that the potential $V$ is bounded and that $\psi\in L^2(\L_L)$  is  in the operator domain
and solves $H_L\psi =E\psi$.  The cube $\L_L$ can be decomposed into unit cubes. Place in each unit
cube arbitrarily  a ball with small, but fixed radius $\delta>0$. Denote the union of the small balls by $S$, see Figure \ref{f:cube}.
\vspace*{1em}

\begin{figure}[h]
\begin{center}
\begin{minipage}{0.5\linewidth}
\begin{tikzpicture}[scale=1.3]
\filldraw[fill=gray] (0.7,0.3) circle (1mm);
\filldraw[fill=gray] (0.5,1.1) circle (1mm);
\filldraw[fill=gray] (0.2,2.3) circle (1mm);
\filldraw[fill=gray] (0.4,3.1) circle (1mm);
\filldraw[fill=gray] (0.8,4.7) circle (1mm);
\filldraw[fill=gray] (1.6,0.2) circle (1mm);
\filldraw[fill=gray] (1.5,1.5) circle (1mm);
\filldraw[fill=gray] (1.2,2.7) circle (1mm);
\filldraw[fill=gray] (1.8,3.3) circle (1mm);
\filldraw[fill=gray] (1.3,4.2) circle (1mm);
\filldraw[fill=gray] (2.2,0.1) circle (1mm);
\filldraw[fill=gray] (2.4,1.3) circle (1mm);
\filldraw[fill=gray] (2.5,2.4) circle (1mm);
\filldraw[fill=gray] (2.8,3.8) circle (1mm);
\filldraw[fill=gray] (2.7,4.5) circle (1mm);
\filldraw[fill=gray] (3.4,0.8) circle (1mm);
\filldraw[fill=gray] (3.5,1.6) circle (1mm);
\filldraw[fill=gray] (3.8,2.6) circle (1mm);
\filldraw[fill=gray] (3.2,3.4) circle (1mm);
\filldraw[fill=gray] (3.7,4.5) circle (1mm);
\filldraw[fill=gray] (4.6,0.2) circle (1mm);
\filldraw[fill=gray] (4.5,1.4) circle (1mm);
\filldraw[fill=gray] (4.3,2.5) circle (1mm);
\filldraw[fill=gray] (4.7,3.7) circle (1mm);
\filldraw[fill=gray] (4.5,4.5) circle (1mm);

\foreach \y in {0,1,2,3,4,5}{
\draw (\y,1) --(\y,2);
\draw (\y,0) --(\y,1);
\draw (\y,2) --(\y,3);
\draw (\y,3) --(\y,4);
\draw (\y,4) --(\y,5);

\draw (0,\y) --(1,\y);
\draw (1,\y) --(2,\y);
\draw (2,\y) --(3,\y);
\draw (3,\y) --(4,\y);
\draw (4,\y) --(5,\y);
}
\end{tikzpicture}
\end{minipage}
 \end{center}
\label{f:cube}
\caption{Cube ${\L_L}$ of size $L$ and the union $S$ of balls.}
 \end{figure}
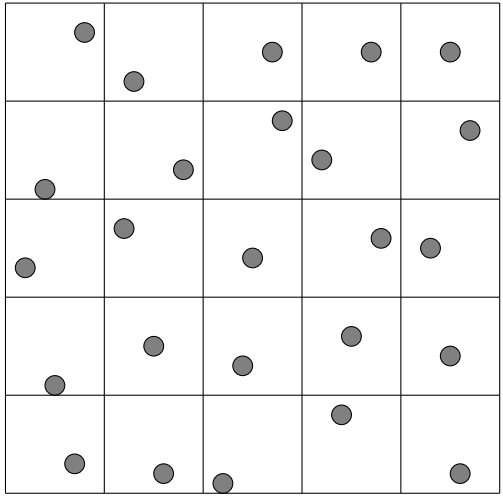

Then $S \subset {\L_L}$ implies obviously
\[
 \int_S |\psi(x)|^2 \leq \int_{\L_L}|\psi(x)|^2
\]

We show that there exists a constant $C$ such that
\begin{equation}
\label{eq:main-intro}
 \int_S |\psi(x)|^2 \geq C \int_{\L_L}|\psi(x)|^2
\end{equation}
The constant is independent of $L \in 2\NN+1$. It depends on the effective potential $V-E$ only through
$\|V-E\|_\infty$. It does not change if we move a small ball in such a way that it is still within the original unit cube.

Of course, the result remains true, if $S$ contains two or more small balls. Thus it applies if the centers of the balls form a Delone set in $\RR^d$.
\medskip

This result is useful for the perturbation theory of eigenvalues. It provides lower bounds on the
derivative of eigenvalues under a perturbation  with positive semidefinite potential which is given by the characteristic function $\chi_S$
of the set $S$.

In particular, one can conclude that the bottom of the spectrum is lifted linearly if the perturbation
$\chi_S$ is switched on by a coupling constant (up to a certain critical value). This allows to conclude a so called uncertainty  principle
for energies near the spectral minimum.
\medskip

In the context of random Schr\"odinger operators one is interested in whether the randomness of the potential
regularizes the  (integrated) density of states. In our situation, where we do not
 impose the usual ergodicity assumptions on the potential,
it is natural to formulate this in terms of spectral properties of a Schr\"odinger operator restricted to a finite cube.
Let ${\L_L}$ be a cube of side $L\in2\NN+1$, $(\Omega, \PP)$ a probability space, $V_0 \colon {\L_L}\to \RR$ a bounded deterministic potential,  $V_\omega \colon {\L_L}\to \RR$ a bounded random
potential and $H_{\omega,L}= -\Delta + V_0+V_\omega$ a random Schr\"odinger operator on ${\L_L}$ with Dirichlet or periodic boundary conditions.
Now one is interested in the continuity properties of the averaged eigenvalue counting function
\[
 n(E,L):= \EE\left\{ \Tr \chi_{(-\infty,E]} (H_{\omega,L}) \right \}
\]
Bounds on the difference  $n(E+\epsilon,L)-n(E,L)$ are called Wegner estimates.
For surveys see, e.g., \cite{KirschM-07} or \cite{Veselic-07b}.
We derive Wegner estimates for random potentials of the type
\[
 V_\omega  = \sum_{k  \in {\tL_L}} \omega_k \, \chi_{B(z_k,\delta)},
 \quad \text{ with } \tL_L:={\L_L}\cap \ZZ^d
\]
 where $(\omega_k)_k$ is an i.i.d.~collection of bounded random variables,
$B(z_k,\delta)$ denotes a ball of radius $\delta>0$ around $z_k$, which is chosen in such a way that
it  is contained in a cube of side one centered at the lattice site $k \in \ZZ^d$, cf. Figure  \ref{f:cube}.
For such random Schr\"odinger operators we derive bounds on $n(E+\epsilon,L)-n(E,L)$
which are proportional to the box volume and
\begin{itemize}
 \item have an optimal dependence on the energy interval size $\epsilon$, for values $E$ near the spectral bottom.
 \item have an optimal dependence on the energy interval size $\epsilon$, up to an additional factor
 $|\log \epsilon|^d$, for other values of $E$.
\end{itemize}
\medskip

For specific types of random Schr\"odinger operators related results have been
derived before in a number of papers.  Let us summarize this briefly.
We will restrict ourselves to papers where the single site potential has \emph{small support}, {i.e.} does not cover the whole unit cube,
otherwise the body of literature would be to large.

First we discuss results for energy intervals near the spectral minimum and in the case  that the non-random part of the Schr\"odinger operator is the pure Laplacian.
Wegner estimates and eigenvalue lifting estimates have been derived in
\cite{Kirsch-96}, in the case that the sequence $\{z_k\}_{k \in \ZZ^d}$ is the lattice $\ZZ^d$ itself. See in particular Lemma 3.1 there,  which is based on control of hitting probabilities of Brownian motion.
A very elementary approach to eigenvalue lifting estimates is provided by the spatial averaging trick, used in \cite{BourgainK-05},
\cite{GerminetHK-07} in periodic situations  and extended to non-periodic situations in \cite{Germinet-08}  and
the forthcoming \cite{GerminetMR}.

A different approach for eigenvalue lifting was derived in \cite{BoutetdeMonvelNSS-06}.
In \cite{BoutetdeMonvelLS-11}  it was shown how one can conclude an uncertainty principle at low energies
based on an eigenvalue lifting estimate.

Related results have been derived for energies near spectral edges in  \cite{KirschSS-98a}, in particular
Proposition 3.2 there, and in \cite{CombesHN-01}.

In one space dimension eigenvalue lifting  results and Wegner estimates have been derived in
\cite{Veselic-96}, \cite{KirschV-02b}. There a periodic arrangement of the sequence  $\{z_k\}_{k \in \ZZ}$ is assumed,
but the proof carries over to the case of non-periodic $\{z_k\}_{k \in \ZZ}$  verbatim. This has been implemented in the context of quantum graphs in \cite{HelmV-07}.

In the case that both the potential $V_0$ and the sequence $\{z_k\}_{k \in \ZZ^d}$  are periodic w.r.t. the lattice $\ZZ^d$,
an uncertainty principle and a Wegner estimate, which are valid for arbitrary energy regions, have been proven in
\cite{CombesHK-03} and \cite{CombesHK-07}.
An alternative proof for the result in \cite{CombesHK-07},  with more explicit control of constants
has been derived in \cite{GerminetK}.
These papers make use of Floquet theory. This is the reason why they are a priori
restricted to periodic background potentials as well as periodic positions of impurities.
In contrast, we use a combinatorial and geometric argument, based on the notion of
\emph{dominating} boxes, in combination with unique continuation estimates as in \cite{BourgainK-05,GerminetK,BourgainK}. For this reason we do not have to assume a periodic situation, but merely an underlying Delone structure of the Hamiltonian.

The case where the background potential is periodic but the impurities need not be periodically arranged has been considered
in \cite{BoutetdeMonvelNSS-06}, \cite{Germinet-08} for low energies.
In the case where the unperturbed operator is a Landau Hamiltonian related results have been obtained in
\cite{CombesHKR-04} and \cite{RojasMolina}.
\medskip

The results derived in this paper are relevant for the theory of Anderson localization,
more precisely, for the so-called multi scale analysis (MSA)
for alloy or Anderson type random Schr\"odinger operators in continuum configuration space.
The MSA is an induction procedure over increasing length scales,
which correspond to sides of larger and larger cubes.
The induction step of the MSA uses in a crucial way a Wegner estimate.
We have already mentioned that such results follow from our Main Theorem  \ref{t:scale_free-UCP}.
This is discussed in detail in  \S \ref{s:Wegner} and \S \ref{ss:optWE}.

The induction anchor of the MSA is based on the so-called initial scale estimate.
Under appropriate assumptions on the distribution of the random variables
the initial scale estimate is implied by explicit quantitative lower bounds on the first eigenvalue
of finite box Hamiltonians.
That the latter follow from our Main Theorem  \ref{t:scale_free-UCP} is discussed in
\S \ref{s:spectral_minimum}.
If the background potential $V_0$ is periodic
an alternative way to establish initial scale estimates
is to use Lifshitz tail bounds. This is carried out in \cite{SchumacherV}.
(Note that Lifshitz tails in the strict sense of the word may not exist,
since the model does not need to be ergodic.)

For all results in Section \ref{s:Anderson-application} it is crucial that in inequality
\eqref{eq:main-intro} (or more precisely in Theorem  \ref{t:scale_free-UCP})
the dependence on the potential enters only through its sup-norm
and is otherwise independent of its particular shape.
\medskip

If the integrated density of states (IDS) exists, i.e.~if the limit
\begin{equation}
\label{eq:convergence-of-expectations}
N(E):=  \lim_{L\to \infty}  n(E,L)
\end{equation}
exists for a dense set of energies $E\in \RR$,
then the Wegner estimates from  \S \ref{s:Wegner} and \S \ref{ss:optWE}
yield bounds on the modulus of continuity of the function $N$,
and imply its local H\"older, or even Lipschitz, continuity.
Note that a special case where \eqref{eq:convergence-of-expectations}
holds is when the IDS is a self-averaging quantity, i.e.{} when
\begin{equation}
\label{eq:a-s-convergence}
N(E):=  \lim_{L\to \infty}  \Tr \chi_{(-\infty,E]} (H_{\omega,L})
\end{equation}
almost surely exists and is $\omega$-independent for a dense set of energies $E\in \RR$.
Relation \eqref{eq:a-s-convergence} can be inferred
from appropriate ergodic theorems. Note that
\eqref{eq:a-s-convergence} is a sufficient, but not necessary condition for
\eqref{eq:convergence-of-expectations} to hold.

It is well known that \eqref{eq:a-s-convergence} holds for the standard alloy type model,
i.e.~for periodically arranged impurities with i.i.d.~coupling constants.
In \cite{GerminetMR} it is established
that \eqref{eq:a-s-convergence} holds for certain more general situations,
e.g.~if the locations of the impurities form a Delone set with a linear repetitivity property.
\medskip

In Section \ref{s:Anderson-application} we discuss two types of
Wegner estimates, one of which is valid for all bounded
intervals along the  energy axis, the other one only near the spectral minimum.

It is quite possible that our method allows also to establish Wegner estimates
with linear dependence on the box volume and energy interval length near spectral edges.
In this context, ideas from \cite{KirschSS-98a} may be relevant.
Note however that for this a consistent notion of spectral edges, valid for an infinite sequence of length scales, is necessary. This is for instance satisfied, if the deterministic background potential $V_0$ is periodic.

The same circle of ideas may be relevant to control the movement of spectral band edges
under perturbation by a Delone-Anderson potential, as has been done for periodic models
in Theorem 2.2 of \cite{KirschSS-98a}. This in turn is a result relevant for deriving the
initial scale estimate of the MSA.
\medskip

The results of the present paper have been announced in the Oberwolfach abstract \cite{Veselic}.
The paper is split into two parts: the two sections which follow are devoted to Theorems on unique continuation
principles, while their proofs are postponed to the end of the paper.
Section \ref{s:Anderson-application} is devoted to applications in the theory of random Schr\"odinger operators.

\section{Scale-free unique continuation principle}
\label{s:scale-free_ucp}
Let $\L_L(x)= [x-L/2,x+L/2]^d$ be a box of sidelength $L$ centered at $x \in \RR^d$ and in particular
$\L_L(0)= [-L/2,L/2]^d$. For brevity we will sometimes write $\L_L:= \L_L(0)$.
For an open $G\subset \RR^d$ let $W^{2,2}(G)$ be the usual second order Sobolev space.
Denote by $\cD(\Delta_{\L,0})\subset W^{2,2}(\L_L)$ (respectively $ \cD(\Delta_{\L,\per})$)
the domain of the Laplacian on $\L=\L_L$ with Dirichlet or periodic boundary conditions, respectively.

Let $V_0\colon \RR^d \to \RR$ be a bounded, measurable  potential, $H_0=-\Delta+V_0$ a Schr\"odinger operator,
$H_{0,L}$ its restriction to ${\L_L}$ with Dirichlet or periodic boundary conditions and $\psi$ an eigenfunction such that
$H_{0,L}\psi = E \psi$ for real $E$. If we set $V:=V_0-E$, $K_V = \|V_0-E\|_\infty$,
and $H_L=-\Delta+ V_0-E=-\Delta+ V=H_{0,L}-E$  on ${\L_L}$,  then $H_{L}\psi= 0$ and $V \colon \RR^d \to [-K_V,K_V]$.
This motivates the hypotheses of the following
\begin{thm}
\label{t:scale_free-UCP}
Fix $K_V\in [0, \infty), \delta \in (0,1]$.
There exists a constant $C_{sfUC}=C_{sfUC}(d,K_V,\delta)>0$ such that,
for any $L\in 2\NN+1$, any sequence
 \begin{equation}
\label{eq:ball_in_box}
Z:=\{z_k\}_{k \in \ZZ^d} \text{ in $ \RR^d$ such that} \quad
B(z_k,\delta) \subset \Lambda_1(k), \quad \text {for all } k \in \ZZ^d
 \end{equation}
any measurable $V\colon {\L_L}\to [-K_V,K_V]$  and any real-valued $\psi\in \cD(\Delta_{\L,0}) \cup \cD(\Delta_{\L,\per})$ satisfying
\begin{equation}
\label{eq:V_psi_conditions}
|\Delta \psi| \leq |V\psi| \quad \text{a.e. on }   \L_L
\end{equation}
 we have
\begin{equation}
\label{eq:scale_free-UCP}
\int_{S_L} |\psi(x)|^2 \ dx
=
\sum_{k \in \tL_L}\norm{\psi}^2_{B(z_k,\delta)}
\geq
C_{sfUC}\norm{\psi}_{\L_L}^2
\end{equation}
where $S_L := S \cap \L_L = \bigcup_{k \in \tL_L} B(z_k,\delta)$ ,
$\tL_L=\Lambda_L \cap \ZZ^d$, and $S:=\bigcup_{k \in \ZZ^d} B(z_k,\delta)$.
\end{thm}

We have an explicit lower bound  on $C_{sfUC}$, namely,
\begin{equation}
\label{eq:CsfUC-dependence}
C_{sfUC}\geq \left( \frac{\delta}{C}\right)^{C+C K_V^{2/3}}
\end{equation}
where $C=C(d)\in(0,\infty)$ is a constant depending on the dimension $d$ only.
In fact one can express $C_{sfUC}$ directly by the constant $C_{qUC}$
from Corollary \ref{c:our_corollary}. In the case of periodic b.c. it reads:
\[
C_{sfUC}
= \frac{1}{2} C_{qUC} \left(d, K_V, R=2 \lceil\sqrt{d}\rceil,\delta, \beta=2(62\lceil\sqrt{d}\rceil)^d\right)
\]
and in the case of Dirichlet b.c.
\[
C_{sfUC}
= \frac{1}{4} C_{qUC} \left(d, K_V, R=2 \lceil\sqrt{d}\rceil,\delta, \beta=2(124\lceil\sqrt{d}\rceil)^d\right)
\]
Note that the constant $C_{sfUC}$ does not depend on  $L\in 2\NN+1$.
It can be chosen uniformly, as long the energy parameter $E$ and the potential $V$ vary in a set,
such that $\|V-E\|_\infty\leq K$ for some fixed $K<\infty$.
The dependence of $C_{sfUC} $ on $\delta$, for small values of
$\delta>0$, is of the same type as in the bound in  Corollary \ref{c:our_corollary}, i.e.~polynomial in $\delta$.\\

In certain applications one is interested in the situation where the set $Z$ is more dilute,
but still commensurate to a sublattice $(M\ZZ)^d$.
In this situation the following Corollary of  Theorem
\ref{t:scale_free-UCP} applies.

\begin{cor}
\label{c:scale_free-UCP}
Fix $K_V\in [0, \infty), M \in \NN, \delta \in (0,1]$.
Let the subset $Z :=\{z_k\}_{k \in (M\ZZ)^d}$ of  $ \RR^d$ be such that
\begin{equation}
\label{eq:M-impurities}
B(z_k,\delta) \subset \Lambda_M(k), \quad \text {for all } k \in (M\ZZ)^d
\end{equation}
Then there exists a constant $C:=C(d)\in(0,\infty)$ depending merely on the dimension $d$, such that,
for any $L\in 2\NN+1$, any measurable $V\colon {\L_{ML}}\to [-K_V,K_V]$  and any real-valued $\phi\in \cD(\Delta_{\L_{ML},0}) \cup \cD(\Delta_{\L_{ML},\per})$ satisfying
$|\Delta \phi| \leq |V\phi| \quad \text{a.e. on }   \L_{ML} $
\begin{equation}
\label{eq:M-scale_free-UCP}
\int_{S_{ML}} |\phi(x)|^2 \ dx
=
\sum_{k \in \tL_{ML}}\norm{\phi}^2_{B(z_k,\delta)}
\geq
\left( \frac{\delta}{C M}\right)^{C+C M^{4/3}K_V^{2/3}} \norm{\phi}_{\L_{ML}}^2
\end{equation}
where $S_{ML} := S \cap \L_{ML} = \bigcup_{k \in \tL_{ML}} B(z_k,\delta)$ ,
$\tL_{ML}=\Lambda_{ML} \cap (M\ZZ)^d$, and $S:=\bigcup_{k \in (M\ZZ)^d} B(z_k,\delta)$.
\end{cor}

There are two ways to prove Corollary \ref{c:scale_free-UCP}.
Either one follows the proof of Theorem \ref{t:scale_free-UCP}
presented in Section \ref{s:proof-scale_free-UCP}
and pays attention to the additional parameter $M$, which basically takes the role of $R$
in equation \eqref{eq:our_corollary}.
This has been carried out in detail in the Appendix of \cite{Rojas-Molina-Thesis}.
Or, one considers the scaled function $\psi(x):=\phi(Mx)$, to which Theorem
\ref{t:scale_free-UCP} can be applied. The resulting inequality for $\psi$ can then be
transformed back to yield \eqref{eq:M-scale_free-UCP}.

Note that if there exists a sequence $\{z_k\}_{k \in \ZZ^d}$ in $ \RR^d$ such that
$z_k  \in \Lambda_1(k)$ for all  $k \in \ZZ^d$, but there is no $\delta>0$ such that
$B(z_k,\delta) \subset \Lambda_1(k)$, for all  $k \in \ZZ^d$, then still
\eqref{eq:M-impurities} holds for $M\geq 3$, $M\in \NN$ and $\delta <1$.

\medskip

An important feature of the scale free unique continuation Theorem \ref{t:scale_free-UCP}
is the quantitative control of the constant $C_{sfUC}$ on the various parameters.
In different applications the different parameters play different roles.
In the Delone-Anderson-model, which we analyze in detail in Section \ref{s:Anderson-application},
the parameter $\delta>0$ is arbitrary but fixed.
Let us give an instance of a random Schr\"odinger operator
where this parameter is a random variable.

\begin{rem}
\label{r:breather}
The dependence of the constant $C_{sfUC}$ on $\delta$ is of particular interest,
if one considers this parameter as variable.
This is  the case for so called random breather potentials
\begin{equation}
\label{eq:breather}
 V_\omega (x) := \sum_{k \in \ZZ^d} \chi_{B(k, r_k(\omega))} (x)
\end{equation}
where $r_k$ is an {i.i.d.} sequence of non-negative, bounded random variables.
Models of this type have been considered in
\cite{CombesHM-96,CombesHN-01,KirschV-10, Veselic-07}.
An uncertainty principle for random breather potentials as in \eqref{eq:breather} has been derived in \cite{SchumacherV}.
\end{rem}

\subsection*{Open question}
Theorem  \ref{t:scale_free-UCP} applies to functions $\psi$ which are pure states in the sense of
quantum mechanics, i.e. eigenfunctions of a Schr\"odinger operator, or more generally, solutions of a
differential inequality.
For applications it would be relevant to know whether the same type of estimate holds for mixed states,
i.e. linear combinations of eigenfunctions, as well, as long as the corresponding eigenvalues range over some bounded energy interval $[-E_0, E_0]$. In that case the parameter $K_V$ would correspond to
$\sup_{E \in [-E_0, E_0]} \|V-E\|_\infty$.

This property can be formulated as an uncertainty principle relating spectral
projectors of a reference operator to a potential.
To formulate this, let $V\colon {\L_L}\to [-K_V,K_V]$ be  measurable,
$H_L=-\Delta+V$ a Schr\"odinger operator on  ${\L_L}$ with periodic or Dirichlet boundary conditions,
$I:= [-E_0, E_0]$ a bounded energy interval as above,
$\chi_I(H_L)$ the spectral projector of $H_L$ onto $I$ and
$W_L$ the multiplication operator by the characteristic function of the set
$S_L := S \cap \L_L = \bigcup_{k \in \tL_L} B(z_k,\delta)$, where $Z:=\{z_k\}_{k \in \ZZ^d}$ is a Delone set  as above.

In this terminology the open question at hand is:
Is there  a constant $C$ depending only on the dimension,
the norm of the (shifted) potential $K:=\sup_{E \in [-E_0, E_0]} \|V-E\|_\infty$,
the radius $\delta$ of the small balls constituting the set $S_L$, such that
\[
\hspace*{3cm} \chi_I(H_L) \, W_L \, \chi_I(H_L) \geq C \, \chi_I(H_L)    \hspace*{3cm} (\clubsuit)
\]
in the sense of quadratic forms.
In the situation where the potential $V$ and the sequence $\{z_k\}_{k \in \ZZ^d}$  are $\ZZ^d$-periodic,
the question has a positive answer.
This has been proven in \cite{CombesHK-03} and \cite{GerminetK}, see also \cite{CombesHK-07}.
An abstract condition, when this property does holds was established in
\cite{BoutetdeMonvelLS-11}, which is applicable in low energy regimes.

We will use this type of uncertainty relation,
in particular theorems from \cite{BoutetdeMonvelLS-11} and \cite{CombesHK-07},
in Section \ref{s:Anderson-application}.
\bigskip

\textbf{Note added in proof}

After the final version of this paper was accepted Abel Klein sent us the preprint 
\emph{Unique continuation principle for spectral projections of Schr\" odinger operators and optimal Wegner estimates for non-ergodic random Schr\" odinger operators},
see  arXiv:1209.4863. There he establishes (among others) the unique continuation principle for spectral projections
$(\clubsuit)$ for energy intervals $I \subset  [-\infty, E_0]$ where $E_0\in \RR$ is fixed and the length 
of $I$ does not exceed $2\gamma$ for an explicitly given $\gamma >0$. 

\section{Quantitative unique continuation principle}
\label{s:quantitative-ucp}

In \cite[Theorem A.1]{GerminetK} a quantitative version of the unique continuation
principle for solutions of the (inhomogeneous) Schr\"odinger equation was derived.
It is an $L^2$-averaged version of the pointwise estimate \cite[Lemma~3.15]{BourgainK-05}.
In the latter result the important parameter was the distance scale $R$ (see Theorem \ref{t:quantitative-UCP} below).

In the proof of Theorem \ref{t:scale_free-UCP}, where the quantitative unique continuation estimate plays an important role
the distance scale is of the order of one, and the dependence on other parameters is more interesting.
We give here a variant of \cite[Theorem A.1]{GerminetK}, as it will be used in Theorem \ref{t:scale_free-UCP}.

At the Oberwolfach Workshop ``Correlations and Interactions for Random Quantum Systems'' in October 2011
we learned that Bourgain and Klein \cite{BourgainK} have derived explicit unique continuation estimates
very similar to the ones presented here.

\begin{thm}
\label{t:quantitative-UCP}
Let $K_V,D_0, R, \beta\in [0, \infty), \delta \in (0,1]$.
There exists a constant $C_{qUC}=C_{qUC}(d,K_V, D_0, R,\delta, \beta) >0$ such that,
for any $G\subset  \RR^d$ open, any $x\in G$, any $\Theta\subset G$ measurable,
satisfying the geometric conditions
 \[
\diam \Theta \leq  R = \dist (x,\Theta), \quad \delta < 4 R,
\quad \text{ and }   \quad B(x, 12R+2D_0)\subset G,	
\]
and any measurable $V\colon G \to [-K_V,K_V]$  and real-valued $\psi\in W^{2,2}(G)$ satisfying the  differential inequality
\begin{equation}
\label{eq:subsolution}
|\Delta \psi| \leq |V\psi| \quad \text{a.e. on }   G
\quad  \text{ as well as } \quad \frac{\norm{\psi}_G^2}{\norm{\psi}_\Theta^2} \leq \beta
\end{equation}
 we have
\begin{equation}
\label{eq:aim}
\norm{\psi}^2_{B(x,\delta)}\geq
C_{qUC}\norm{\psi}_\Theta^2
\end{equation}
\end{thm}

The differential inequality in \eqref{eq:subsolution} is in particular satisfied for any real eigensolution
\begin{equation}
\label{eq:eigenfunction}
-\Delta \psi+V\psi=0 \quad \mbox{a.e. on }G
\end{equation}
In our application $\psi$ arises as the restriction of an eigenfunction of a selfadjoint Schr\"odinger operator
(on a bounded domain containing $G$) with real potential. Thus there is no loss of generality to assume that the eigenfunction is real-valued as well.

If $-\Delta \psi+V\psi=E\psi$ is an eigenfunction for a non-zero eigenvalue, we can
subsume it into the potential $W:= V-E$ and obtain the analog of estimate \eqref{eq:aim}
where the parameter $K_V$ is replaced by $K_{W} = K_{V-E}$.\\[1em]

The bound \eqref{eq:subsolution} is non-linear in the sense that the constant $C_{qUC}$
depends on the (sub)solution $\psi$ through the upper bound $\beta$ on the ratio of local $L^2$-norms.
However, the constant $C_{qUC}$ is uniformly bounded if $\beta$ varies only over a compact subset of $(0,\infty)$.
This follows from the functional dependence of $C_{qUC}$ on the parameters, which we discuss next.\\

It is possible to give an explicit bound on $C_{qUC}$, namely
\begin{equation}
\label{eq:CqUC}
C_{qUC}(d,K_V, D_0, R,\delta, \beta)\geq
\frac{5}{16} \frac{1}{41} \frac{C_2^3}{C_3 K_\Delta^4} \frac{\delta^4}{R^2} \left(\frac{\delta}{48R}\right)^{2\alpha} \frac{1}{1+K_V^2}
\end{equation}
where $C_2,C_3$ are constants depending only on the dimension arising from the Carleman estimate,
 $K_{\Delta}$  is a constant arising in the choice of the differentiable cut off function $\eta$ in \eqref{eq:eta}.
Here $\alpha$ has to satisfy the lower bounds
\begin{align*}
\alpha &\geq C_2
\\
\alpha &\geq (24^5 C_3 K_V^2 R^4)^\frac{1}{3}
\\2\alpha &\geq \ln\left(\left(\frac{24RK_\Delta}{D_0}\right)^4 C_3 \frac{1+K_V^2}{D_1^2} \beta\right)
\quad \text{ where } D_1=\min(D_0,1)
\end{align*}

This explicit functional dependence shows in particular that the lower bound \eqref{eq:CqUC} on $C_{qUC}$ is
an increasing, polynomial function of $\delta\in(0,\infty)$,
an increasing, polynomial function of $1/\beta\in(0,\infty)$,
an increasing, polynomial function of $D_0\in(0,\infty)$,
a decreasing, exponential function of $K_V\in[0,\infty)$,
and a decreasing, exponential function of $R\in(0,\infty)$.
By 'polynomial' dependence we mean a power-law dependence with exponent $\geq 1$.
The dependence of the lower bound \eqref{eq:CqUC} on $C_{qUC}$ on each of the parameters is continuous.
Thus there is an uniform lower bound on $C_{qUC}$ as long as the parameters vary over a compact set.
\medskip

To make the dependence of $C_{qUC}(d,K_V,D_0,R,\delta,\beta)$ on the various parameters
more transparent we consider a somewhat specialised geometric situation in the next

\begin{cor}
\label{c:our_corollary}
Let $K_V, R, \beta\in [0, \infty), \delta \in (0,1]$.
Let $G\subset  \RR^d$  open, $x\in G$, $\Theta\subset G$ measurable,
satisfy the geometric conditions
 \begin{equation}
 \label{eq:geometric_conditions}
\diam \Theta \leq  R = \dist (x,\Theta),
\ \sqrt{d}\leq R, \text{ and }   \ B(x, 14R)\subset G,
 \end{equation}
and $V\colon G \to [-K_V,K_V]$ measurable, $\psi\in W^{2,2}(G)$ real-valued, satisfy
\begin{equation}
|\Delta \psi| \leq |V\psi| \quad \text{a.e. on }   G
\quad  \text{ and } \quad \frac{\norm{\psi}_G^2}{\norm{\psi}_\Theta^2} \leq \beta
\end{equation}
Then there exists a constant ${C}={C}(d)\in (1,\infty)$ depending only on the dimension, such that
\begin{align}
\label{eq:our_corollary}
\norm{\psi}^2_{B(x,\delta)}
&\geq
C_{qUC}
\norm{\psi}_\Theta^2
\ \text { where } \
C_{qUC}:=
\left(\frac{\delta}{{C R}}\right)^{{C} + {C} R^{4/3}K_V^{2/3} + \ln \beta}
\end{align}
If we assume additionally that $\dist(x,\Theta)$ is bounded by $2\lceil\sqrt{d}\rceil$
(This is the case in the  proof of Theorem \ref{t:scale_free-UCP}.),
\eqref{eq:our_corollary} simplifies to
\begin{align}
\norm{\psi}^2_{B(x,\delta)}
&\geq
\left(\frac{\delta}{{C}}\right)^{{C} + {C} K_V^{2/3} + \ln \beta}
\norm{\psi}_\Theta^2
\end{align}
where the constant $C\in(0,\infty)$ depends on the dimension $d$ only.
\end{cor}

In the  proof of Theorem \ref{t:scale_free-UCP} we will be able tu use a $\beta$ which is a constant depending on the dimension $d$ only.

\section{Application of the UCP to Delone-Anderson models}
\label{s:Anderson-application}

While the results on unique continuation from Sections \ref{s:scale-free_ucp}
and \ref{s:quantitative-ucp} are quite general in nature, we now turn to applications
in the spectral theory of random Schr\"odinger operators.
From the physical point of view one, or even the, ultimate goal is to understand the transport
properties of the solid state which is modeled by the
random operator. The mathematical analysis of the model has to start with quite basics questions, like:
where is the spectrum located, how sensitive is it under perturbations, how do symmetries of the model translate into spectral properties, what are a priori bounds on eigenfunctions and resolvents etc.
These results are then complemented with probabilistic and combinatorial arguments
to yield statements about localization properties of the model.

The questions which are addressed in this section  concern the location
of the spectral minimum and the smearing-out of eigenvalues of random operators on a finite cube
by the average over random variables entering the model.

Usual alloy or Anderson-type potentials have a built-in ergodicity property.
We consider here random potentials which have a similar structure, but lack
the stochastic translation invariance.

\subsection{The Delone-Anderson model}
\label{s:DAmodel}

\begin{dfn}
A set $Z\subset \RR^d$ is a \emph{Delone} set if
there exist positive real numbers $\tilde M<M$, $M\in\NN$  such that
for any $x\in\RR^d$, we have $\sharp \{\DeloneSet \cap \L_{\tilde M}(x)\} \leq 1$ and
$\sharp \{ \DeloneSet \cap \L_{M}(x)\} \geq 1$,
where $\sharp$ stands for cardinality.
\end{dfn}

This implies that for the scale $M \in \NN$, each periodicity cell of the lattice $(M\ZZ)^d$
contains at least one element of $\DeloneSet$. Moreover, no periodicity cell contains more than
$\tilde N=\lceil M/\tilde M\rceil^d$ elements of $\DeloneSet$.

Note that, by definition, we can always write $Z=Z_{\Gamma_1}\cup Z_{\Gamma_2}$
as a disjoint union,
of a subset $Z_{\Gamma_1}$  which contains exactly one point in each periodicity cell of the
lattice $(M\ZZ)^d$, and $Z_{\Gamma_2}:= Z\setminus Z_{\Gamma_1}$
which intersects each periodicity cell in at most $\tilde N-1$ points.
It is then natural to enumerate the elements of $Z_{\Gamma_1}=\{z_j\}_{j\in \Gamma_1}$ by
the index set $\Gamma_1:=(M\ZZ)^d$. Similarly, the elements of
$Z_{\Gamma_2}=\{z_j\}_{j\in \Gamma_2}$ are enumerated by  some countable set $\Gamma_2$
(disjoint to $\Gamma_1$). For the Delone set $Z=Z_{\Gamma_1}\cup Z_{\Gamma_2}$  itself
we obtain the enumeration $Z=\{z_j\}_{j \in \Gamma}$ with $\Gamma= \Gamma_1 \cup \Gamma_2$.

Now, let us consider a random Schr\"odinger operator of the form $H_\omega=H_0+V_\omega$, where $H_0=-\Delta+V_0$,
and $V_0\colon \RR^d \to \RR$ is a bounded and measurable potential.  The Delone-Anderson potential is given by
\begin{equation}
 V_\omega(x):= \sum_{j \in\Gamma} \ \omega_j u_j(x)
\end{equation}
where $\Gamma$ is the index set of a Delone set $Z$, and we have the following
\begin{hyp}
\label{hyp:Delone_Anderson}~
\begin{enumerate}[(i)]
\item
The random variables $\omega_j, j\in \Gamma,$ are independent with probability distributions $\mu_j$, such that
for some $m>0$, $ \supp \mu_j \subset [-m, m]$, $\forall j\in \Gamma$.
\item
Let $ s\colon [0,\infty) \to[0,1]$ be the global modulus of continuity of the family
$\{\mu_j\}_{j\in \Gamma_1}$, that is,
\begin{equation}
\label{definition-s-mu-epsilon}
\forall j \in \Gamma_1: \quad
\sup\Big\{\mu_j\Big(\Big[E-\frac{\epsilon}{2},E+\frac{\epsilon}{2}\Big]\Big) \, \Big| \, E \in \RR\Big\}
\leq s(\epsilon)
\end{equation}
\item
Let $0 < \delta_- < \delta_+<\infty$, $0 < C_- \leq C_+ <\infty$.
The sequence of measurable functions $u_j \colon \RR^d \to \RR, j \in \Gamma$ is
such that
\begin{align*}
\forall j \in \Gamma_1:
& \quad  C_- \chi_{B(z_j,\delta_-)} \leq u_j \leq C_+ \chi_{B(z_j,\delta_+)}
\\
& \quad  \text{ and } B(z_j,\delta_-) \subset \L_M(j)
\\
 \forall j \in \Gamma_2:& \quad  |u_j| \leq C_+ \chi_{B(z_j,\delta_+)}
 \end{align*}
\end{enumerate}
\end{hyp}

Then, we can decompose the random potential as
\begin{equation}
\label{eq:decompse_V}
 \sum_{j \in\Gamma} \omega_j u_j(x)
= \sum_{j \in\Gamma_1} \omega_j u_j(x) +\sum_{j \in\Gamma_2} \omega_j u_j(x)
\end{equation}
In the proofs we will employ the first part of the potential in \eqref{eq:decompse_V}
to lift eigenvalues. Of the second part we will only use the fact that it is a bounded function, uniformly  in $\omega$.
It can be subsumed into the background potential $V_0$, since the random variables entering in $ V_\omega$ are independent.

For simplicity we will assume in the proofs of Theorems \ref{t:Wegner}, \ref{t:uncertainty}, and \ref{t:optWE} that $M=1$.
The statement of all  theorems is still true for $M>1$. If one is interested in the dependence on the parameter $M$,
Corollary \ref{c:scale_free-UCP} gives its influence on the constant $C_{sfUC}$, from which the constant $C_W$ in the Wegner estimate follows, see for instance \eqref{eq:CWdependence}. In the case of Theorem  \ref{t:optWE}  this requires going through the proof in
\cite{CombesHK-07}.

Delone-Anderson operators are rather general and specific cases are given by the following more familiar
examples
\begin{exm}
\begin{description}
 \item[Standard alloy type potential]
Let $\omega_j, j \in \ZZ^d$ be an i.i.d. sequence of bounded random variables,
and $u$ a fixed  bounded single site potential of compact support such that for some $C_-, \delta_->0$
\[
 u \geq C_- \chi_{B(0,\delta_-)}
\]
Then the alloy type potential
\[
 V_\omega(x) := \sum_{j \in \ZZ^d} \omega_j \ u (x-j)
\]
satisfies the assumptions of our model with $M=1$.
  \item[Alloy type potential with random displacements]
Let $\omega_j, j \in \ZZ^d$ and $u$ be as above and $\xi_j, j \in \ZZ^d$
an i.i.d. sequence of random variables, which are independent of the sequence $\omega_j, j \in \ZZ^d$,
and take values in some ball $B(0, \tilde R) \subset \RR^d$. For each fixed configuration $\{\xi_j\}_{j \in \ZZ^d}$
the alloy-displacement potential
\[
 V_{\omega, \xi} (x) := \sum_{j \in \ZZ^d} \omega_j \ u (x-j- \xi_j)
\]
satisfies the Assumption \ref{hyp:Delone_Anderson} of our model, in particular (iii),
if   $2(\tilde R+\delta_-)\leq M$.
\end{description}

\end{exm}

Denote by $H_{0,L,x}$, $H_{\omega,L,x}$ the finite volume operators obtained by restricting $H_0$ and $H_\omega$,
respectively, to the box $\L_L(x)$ with either Dirichlet or periodic boundary conditions.

\subsection{Wegner estimate for Delone-Anderson models at all energies}
\label{s:Wegner}

Our first result on the model described in the last paragraph is a
Wegner estimate which is valid for all bounded intervals on the energy axis.
\begin{thm}
\label{t:Wegner}
Let $H_\omega$ be a random Schr\"odinger operator as in Assumption \ref{hyp:Delone_Anderson}.
Then for each $E_0\in \RR$ there exists a constant $C_W$ such that, for all $E\le E_0$ and
$\epsilon \le 1/3$, all $L\in 2\NN+1$ and $x \in \RR^d$
\begin{equation}
\label{e-WE}
\EE\{\Tr [ \chi_{[E-\epsilon,E+\epsilon]}(H_{\omega, L,x}) ]\}
\le C_W \ s(\epsilon) \, |\ln \, \epsilon|^d \ |\Lambda_L|
\end{equation}
\end{thm}

Let us discuss the dependence of the Wegner constant $C_W$ in detail.
\begin{itemize}
\item
 The constant $C_W$ depends only on $E_0$, $\|V_0\|_\infty$, $m$, $\tilde N$, $C_-$, $C_+$, $\delta_-$,
 $\delta_+$ and $M$.
 \item
 To understand the dependence of $C_W$ on these various parameters,
let us assume for the moment
that all potentials $V_0$ and $u_j$ and random variables $\omega_j$ entering the model are non-negative.
Then
\begin{equation}
\label{eq:CWdependence}
 C_W=C_{E_0} \,  \cdot \,  \left\lceil \frac{4}{C_-\, C_{sfUC}} \right\rceil
\end{equation}
 where $C_{E}= K_1 {\rm e} ^{E+1}+ 2^d K_2$ is the constant from Corollary \ref{c-SSFbound}
which depends only on the dimension $d$, the energy bound $E_0$, and on $\delta_+$,
and $C_{sfUC}$ is the constant from the scale-free unique continuation principle, Theorem  \ref{t:scale_free-UCP}.
\item
If we allow now $V_0$ and $u_j, j\in \Gamma_2$ to take on both signs, the constant $C_E$ will additionally depend on
\[
 \sup_\omega \sup_x (V_0+ V_\omega)^-(x)
\]
 where $f^-:= -\min (0, f)$.
\end{itemize}
\subsection{Proof of the Wegner estimate \ref{t:Wegner}}
Recall that by the Assumption \ref{hyp:Delone_Anderson} the potential is uniformly bounded in the randomness
and the space coordinate $x$.
Thus we may  without loss of generality assume in the proof that
\[
 \inf_\omega \inf_x \left( V_0(x) + V_\omega (x) \right)\geq 1
\]
Thus, even if we substract a small value $\epsilon>0$ from the potential, it will remain non-negative.

For the reader's convenience we recall  certain tools from \cite{HundertmarkKNSV-06}
which will be used in the proof.  The first is a partial integration formula
  \cite[Lemma 6]{HundertmarkKNSV-06} for singular distributions:

\begin{lem}
\label{l-sme}
Let $\mu$ be a probability measure with support in $(a,b)$,
$\phi\in C^1(\RR)$ be a non-decreasing, bounded function,
and $s(\epsilon)$ as defined in \eqref{definition-s-mu-epsilon}.
Then for any $\epsilon>0$,
\[
\int_\RR [\phi(\lambda+\epsilon)-\phi(\lambda)] \, d\mu(\lambda)
\leq s( \epsilon)\cdot [\phi(b+\epsilon)-\phi(a)]
\]
\end{lem}

Next, we state a bound on the spectral shift function.
Let $H_0$ be the Laplace operator on the cube $\Lambda_L(x)$ with Dirichlet or periodic boundary conditions,
$V$ a bounded, non-negative potential, and $u$ a compactly supported non-negative single-site potential.
Set $H_1:=H_0+V$ and $H_2=H_1+u$. Both Schr\"odinger operators are selfadjoint and
lower-semibounded with purely discrete spectrum.
We define the \emph{spectral shift function} $\xi\colon \RR\to\RR$ associated to  the pair of operators $H_1,H_2$ by
\begin{multline}
\xi(\lambda)= \sharp \{\text{eigenvalues of $H_2$ not exceeding } \lambda\}
\\- \sharp \{\text{eigenvalues of $H_1$ not exceeding } \lambda \}
\end{multline}
It satisfies the Lifshitz-Krein's trace identity
\begin{equation}
\label{eq:SSF}
\Tr [\rho(H_2)-\rho(H_1)]=\int \rho'(\lambda)\xi(\lambda)d\lambda
\end{equation}
for all $\rho\in C^\infty$ with compactly supported derivative.
Note that the left hand side of  \eqref{eq:SSF} reduces to a sum over finitely many eigenvalues
since the spectrum is discrete and $\supp \rho'$ compact.
Theorem 2 in \cite{HundertmarkKNSV-06} states:
\begin{thm}
Let $\xi$ be the spectral shift function for the pair $H_1,H_2$.  There exist constants $K_1,K_2$
depending only on $d$ and  $\diam \supp u$, such that for any measurable, bounded $f\colon\RR\to\RR$ with $\supp f \subset [a,b]$
\begin{equation}
\int f(\lambda)\xi(\lambda)d\lambda \leq K_1 e^b+K_2\left( \ln(1+\norm{f}_\infty)^d \right) \norm{f}_1
\end{equation}
\end{thm}

Let  $\epsilon \in (0,1/3]$ and $\rho_\epsilon$ be a smooth, non-decreasing function
such that on $(-\infty, -\epsilon]$ it is identically equal to $-1$,
on $[\epsilon, \infty)$ it is identically equal to zero, and $\|\rho_\epsilon'\|_\infty \le 1/\epsilon$.
The previous theorem combined with Lifshitz-Krein's trace identity gives the following

\begin{cor}
\label{c-SSFbound}
Let $H_1,H_2$ and $\rho_\epsilon$ be as above. There is a constant $C_E$
depending only on $E$, $d$, $\diam \, \supp \, u$, such that
\begin{equation} \label{e-SSF}
\Tr \left [ \rho_\epsilon (H_2)-\rho_\epsilon (H_1)  \right  ] \le C_E \, |\ln \epsilon|^d
\end{equation}
The constant  $C_E $ can be chosen equal to $ K_1 {\rm e} ^{E+1}+ 2^d K_2$
with $K_1, K_2$ from the last Theorem.
\end{cor}

\begin{proof}[Proof of Theorem \ref{t:Wegner}]
Let  $\rho:=\rho_\epsilon$ be as before. Then
\[
\chi_{[E-\epsilon,E+\epsilon ]} (x)
\le \rho(x-E+2\epsilon) -\rho(x-E-2\epsilon)
\]
In the remainder of the proof we will suppress the box center $x$ in the notation.
All estimates which follow are uniform w.r.t. $x\in \RR^d$.
Let $\lambda_n^L(\omega)$ denote the eigenvalues of $H_{\omega,L}$
enumerated in non-decreasing order and counting multiplicities
and $\psi_n$ the  normalised eigenvectors corresponding to $\lambda_n^L(\omega)$.
Then
\[
\lambda_n^L(\omega)
=
\la \psi_n, H_{\omega,L} \psi_n\ra
=
\int_{\L_L}dx \, \bar \psi_n (H_{\omega,L} \psi_n )
\]
Define the vector $ e=(e_j)_{j\in\Gamma}$ by  $e_j=1$ for $j\in\Gamma_1$  and $e_j=0$ for $j\in\Gamma_2$.
Consider the monotone shift of $V_\omega$
\[
 V_{\omega+ \epsilon \cdot e}
= \sum_{j \in\Gamma_1} (\omega_j+ \epsilon ) u_j
+
 \sum_{j \in\Gamma_2} \omega_j  u_j
\]
and recall that $\tL_L= \Lambda \cap \Gamma_1$.
By first order perturbation theory or the Hellmann-Feynman-theorem
\[
 \frac{d }{d\epsilon} \lambda_n^L(\omega+ \epsilon \cdot e)
= \langle \psi_n, \sum_{j \in \tL}  u_j \,   \psi_n \ra
\]
By the scale free unique continuation
principle there exists a constant $C_{sfUC}$ depending on the energy $E_0$,
$\delta_-$ and the overall supremum $\sup_\omega \sup _x |V_{0}(x) +V_\omega(x)|  $
of the potential such that
\begin{equation*}
\sum_{j \in \tL}  \la \psi_n, u_j \,  \psi_n \ra
\geq
C_- \sum_{j \in \tL}\la \psi_n,  \chi_{B(z_j,\delta_-)}\psi_n \ra \geq
C_-\cdot C_{sfUC}
=: \kappa
\end{equation*}
where we used that $\norm{\psi}_{\Lambda}^2=1$.
Now
\begin{align*}
\lambda_n^L(\omega+ \epsilon \cdot e)
&=
\lambda_n^L(\omega) + \int_0^\epsilon dt \,  \frac{d \lambda_n^L(\omega+ t \cdot e) }{d t}
\\
&\geq
\lambda_n^L(\omega) + \int_0^\epsilon dt \,  \kappa  = \lambda_n^L(\omega) + \epsilon \, \kappa
\end{align*}
Since $\rho$ is monotone
\[
 \rho (\lambda_n^L(\omega+ \epsilon \cdot e) )
\geq
\rho (\lambda_n^L(\omega) + \epsilon \, \kappa )
\]
We calculate the trace using eigenvalues.
\begin{align}
 \label{e-shift-bound}
\Tr  [\rho (H_{\omega, L}+ \epsilon\kappa)]
&= \sum_{n} \rho(\lambda_n^L(\omega)+ \epsilon\kappa)
\leq
 \sum_{n} \rho(\lambda_n^L( \omega + \epsilon e))
\\  \nn
&=
\Tr\Big[\rho (H_{\omega+ \epsilon\cdot e, L})\Big]
=
\Tr\Big[\rho (H_{\omega, L}+ \epsilon\sum_{k \in \tL} u_k ) \Big]
\end{align}
By writing $\epsilon_\kappa=\epsilon/\kappa$, we have
\begin{equation}
\Tr  [\rho (H_{\omega,L}+ \epsilon)]
\le \Tr\Big[\rho (H_{\omega,L}+ \epsilon_\kappa\sum_{k \in \tL} u_k ) \Big]
\end{equation}
Since $L\in 2\NN+1$, the box $\Lambda_L$ is decomposed in $N:=L^d$ unit cubes. We
enumerate the sites in $\Lambda_L$ by $ k\colon \{1, \dots, N\} \to \tL_L = \Lambda_L \cap
\Gamma_1$, $n\mapsto k(n)$ and set
\[
W_0 \equiv 0, \quad  W_n =\sum_{m=1}^{n} u_{k(m)}, \qquad n=1,2,\dots, N
\]
Thus
\begin{equation}
\label{e-project}
\begin{split}
&
\EE \{\Tr [\chi_{[E-\epsilon,E+\epsilon]} (H_{\omega,L}) ]\}\\
 & \le
\EE \{\Tr [ \rho(H_{\omega,L}+2\epsilon)-\rho(H_{\omega,L}-2\epsilon)]\}
\\
& =\EE \{\Tr [ \rho(H_{\omega,L}-2\epsilon+4\epsilon)-\rho(H_{\omega,L}-2\epsilon)]\}
\\
& \le
\EE \{\Tr [ \rho(H_{\omega,L}-2\epsilon+4\epsilon_\kappa W_{L})-\rho(H_{\omega,L}-2\epsilon)]\}
\\    	
 &\le
\EE \left \{\sum_{n=1}^{N}\Tr [ \rho(H_{\omega,L}-2\epsilon+4\epsilon_\kappa W_{n})-
\rho(H_{\omega,L}-2\epsilon+4\epsilon_\kappa W_{n-1})] \right
\}
 \end{split}
\end{equation}
We fix $n \in \{1, \dots, N\}$, denote  
\[
\omega^\perp := \{\omega_k^\perp\}_{k \in \tL}, \qquad
\omega_k^\perp :=\begin{cases} 0 \quad &\text{if $k=k(n)$}, \\
\omega_k \quad &\text{if $k\neq k(n)$}, \end{cases}
\]
and set
\[
\phi_n(t) := \Tr\big[\rho\big(H_{\omega,L}^\perp-2\epsilon +4\epsilon_\kappa W_{n-1}+t \cdot
u_{k(n)}\big)\big], \quad t\in\RR.
\]
The function $\phi_n$ is continuously differentiable, monotone increasing and bounded.
By definition of $\phi_n$,
\begin{multline*}
\EE \{\Tr [ \rho(H_{\omega,L}-2\epsilon+4\epsilon_\kappa  W_n))
-\rho(H_{\omega,L}-2\epsilon+4\epsilon_\kappa  W_{n-1})]\}
\\  \le
\EE \{\int [\phi_n(\omega_{k(n)}+4\epsilon_\kappa)-\phi_n(\omega_{k(n)})]\, d\mu(\omega_{k(n)})\}
\end{multline*}
Let $a=(\inf\supp \mu)-1$ and $b= (\sup\supp \mu)+1$.
Using Lemma \ref{l-sme} together with Corollary \ref{c-SSFbound},  we have
\begin{align*}
\int [\phi_n(\omega_{k(n)}+4\epsilon_\kappa)-\phi_n(\omega_{k(n)})]\, d\mu(\omega_{k(n)})
&\le
s(4\epsilon_\kappa) [\phi_n(b+4\epsilon_\kappa)-\phi_n(a)]
\\
&\le C_E \, s(4\epsilon_\kappa) \left( \ln(1/\epsilon)\right)^d
\end{align*}
which implies that \eqref{e-project} is bounded by
\begin{align*}
C_E \sum_{n=1}^{N} s(4\epsilon_\kappa) \left( \ln(1/\epsilon)\right)^d
\leq
C_E  \, s\left(4\epsilon/\kappa\right) \left( \ln(1/\epsilon)\right)^d \, L^d
\end{align*}
where we apply Corollary \ref{c-SSFbound} successively $N$ times.
Due to monotonicity  and additivity of measures we have
\[
s(\epsilon) \leq s( t\epsilon) \leq s( \ceil{t}\epsilon) \leq \ceil{t} s( \epsilon)
\text{  for every } t\in [1,\infty)	
\]
Thus we obtain
\begin{equation}
\EE\{\Tr [ \chi_{[E-\epsilon,E+\epsilon]}(H_{\omega, L}) ]\}
\le
C_E  \ceil{4/\kappa} \, s(\epsilon) \left( \ln(1/\epsilon)\right)^d \, L^d	
\end{equation}

\end{proof}

\subsection{Perturbation of the spectral minimum}
\label{s:spectral_minimum}
We derive a lower bound on the shift of the lowest eigenvalue of a box Hamiltonian.
This has several important consequences.
\begin{itemize}
\item
The first one is an estimate on spectral projectors, which is sometimes called uncertainty principle.
It is given in the present section.
\item
From this an improved Wegner estimate for low energies, presented in the following \S\ref{ss:optWE},
can be derived.
\item
Furthermore, the quantitative lower bound on the lifting of the first eigenvalue of box Hamiltonians can be used to derive initial scale estimates for the MSA of Delone-Anderson models. This is pursued in \cite{Rojas-Molina-Thesis} and \cite{GerminetMR}.
\end{itemize}

To obtain the mentioned uncertainty principle for spectral projectors, we will use a result of
\cite{BoutetdeMonvelLS-11} already mentioned in the introduction.
A convenient formulation is:

\begin{thm}[Theorem 1.1 in \cite{BoutetdeMonvelLS-11}]
 For $t_0>0, t\in [0,t_0]$ let $H_t=H_0+tW$ be a self-adjoint, lower semi-bounded operator on a Hilbert space. Set $\lambda(t)=\min \sigma (H_t).$
Assume there exists $\kappa>0$ such that
\begin{equation*}
 \forall t\in [0,t_0): \qquad \lambda(t) \geq \lambda(0) +\kappa~t
\end{equation*}
Then for any $q\in (0,1)$
\begin{equation*}
 \forall t\in [0,t_0): \qquad
\chi_{I(t)} (H_0)W \chi_{I(t)} (H_0)
\geq (1-q) \kappa  \, \chi_{I(t)} (H_0)
\end{equation*}
where $\chi_{I(t)} (H_0)$ is the spectral projector on $I(t)=(-\infty, \lambda(0)+q\kappa\cdot t]$.
\end{thm}

Now follows the result of this section:

\begin{thm}
\label{t:uncertainty}
Let $V_0\colon \RR^d \to \RR$ be  bounded and measurable, $t,\delta \in (0,1]$, $\{z_k\}_{k\in \ZZ^d} \subset \RR^d$ a sequence such that
\[
 \forall \ k\in \ZZ^d: \quad B(z_k, \delta) \subset \Lambda_1(k)
\]
Let $\chi$ denote the characteristic function of $\bigcup_{k \in \ZZ^d} B(z_k,\delta) $
and $W\geq C_- \cdot \chi$ a bounded potential with $C_- >0$.
\begin{enumerate}[(a)]
 \item
For $L\in 2\NN+1$ and $x\in \RR^d$, denote by $\lambda^{L,x}(t)=\inf \sigma(H_{t,L, x})$ the bottom of the spectrum of
$H_{t,L,x}:= -\Delta + V_0 +t W$ on $\Lambda_L(x)$ with periodic or Dirichlet boundary conditions.
Then
\[
\forall  \ t \in (0,1]: \quad \lambda^{L,x}(t) \geq \lambda^{L,x}(0) + \kappa \cdot t
 \]
where $\kappa:= C_- \cdot C_{sfUC} $ and $C_{sfUC}$ is the constant from the scale-free unique continuation principle.
\item
Fix $q\in (0,1)$ and set $I=(-\infty,\lambda^{L,x}(0) +q \kappa]$.
Then the following uncertainty principle holds
\[
\chi_I(H_{0,L,x}) W \chi_I(H_{0,L,x}) \geq (1-q)\kappa \chi_I(H_{0,L,x})
\]
Here we used for the restriction of $W$ to $\Lambda_L(x)$ the same symbol.
\end{enumerate}

\end{thm}
\begin{proof}
The normalized ground state $\psi(t)$ of $H_{t,L,x}$ satisfies
\begin{equation}
\label{eq:e-v-lifting}
\lambda^{L,x}(t)=\inf \sigma(H_{t,L,x}) =\langle \psi(t), H_{t,L,x}\psi(t) \rangle
\end{equation}
We want to apply the scale free unique continuation principle from Section \ref{s:scale-free_ucp} to \eqref{eq:e-v-lifting}.
We have to do this with some care since we will be dealing not with one fixed potential,
but a whole family of them. Likewise we will not be dealing with one single energy which features as an eigenvalue.  Now since the constant
$C_{sfUC}$ depends on the potential  (actually, on the \emph{effective} potential which is the difference of the potential and the eigenvalue) it is important to make sure that it is uniformly bounded away form zero,
for the family of effective potentials in question. This is actually the case since
$C_{sfUC}$ depends on the potential $V$ only through the sup-norm $K_V$.

Note that the family of potentials  $V_0+tW$, $t \in [0,1]$  is uniformly bounded.
Thus the set $\{ \lambda^{L,x}(t):\,t\in [0,1],L\in 2\NN+1\}$ is contained in a compact interval $\cI \subset\mathbb R$.
We will apply the scale-free unique continuation principle to this family of potentials.
Note that the parameter $K_V$  which enters the constant $C_{sfUC}$ is bounded uniformly by the finite number
$\sup \{\norm{V_0+tW-E}_\infty;\, E\in \cI, \,t\in[0,1] \}$.
 It follows
\begin{align}
\label{ev_eq}
\lambda^{L,x}(t)&=\langle \psi(t), H_{t,L,x}\psi(t) \rangle  = \langle \psi(t), H_{0,L,x}\psi(t) \rangle + t\langle \psi(t),W \psi(t) \rangle\\
& \geq \langle \psi(t), H_{0,L,x}\psi(t) \rangle + C_-\cdot C_{sfUC} \, t
\\
&\geq  \lambda^{L,x}(0) +t \kappa
\end{align}
where in the last line we used the min-max principle on  the operator $H_{0,L,x}$.
In particular, for $t=1$,
\begin{equation}
\label{lift} \lambda^{L,x}(1)\geq \lambda^{L,x}(0) + \kappa.
\end{equation}
Now, consider the interval $J=[\lambda^{L,x}(0),\lambda^{L,x}(0) +q \kappa]$, for some $q\in (0,1)$
and note that $\chi_I(H_{0,L,x})= \chi_{J}(H_{0,L,x})$.	  We can then
apply  \cite[Theorem 1.1]{BoutetdeMonvelLS-11} and conclude that
\begin{align}
\chi_I(H_{0,L,x}) \, W \,  \chi_I(H_{0,L,x}) & \geq (\lambda^{L,x}(1)-\lambda^{L,x}(0) -q \kappa) \, \chi_I(H_{0,L,x}) \\
& \geq (1-q)\kappa \,  \chi_I(H_{0,L,x})
\end{align}
 \end{proof}
In the case of Dirichlet boundary conditions we know that for any $x\in\mathbb R^d$ and $L\in \mathbb N$
\begin{equation}
\label{eq:bottom-comparison}
  \lambda^{L,x}(0)\geq E_{\min} := \inf \sigma(-\Delta+V_0)
\end{equation}

\begin{cor}
Let the hypotheses and notation of Theorem \ref{t:uncertainty}  hold and assume that Dirichlet boundary conditions are imposed.
Then
\[
\forall  \ t \in (0,1]: \quad \lambda^{L,x}(t) \geq E_{\min} + \kappa \cdot t
 \]
Fix $q\in (0,1)$ and set $I=(-\infty,E_{\min} +q \kappa]$.
Then the following uncertainty principle holds
\[
\chi_I(H_{0,L,x}) W \chi_I(H_{0,L,x}) \geq (1-q)\kappa \chi_I(H_{0,L,x})
\]
\end{cor}

\subsection[Optimal Wegner estimate at low energies]{Optimal Wegner estimate at low energies for Delone-Anderson models}
\label{ss:optWE}

For sufficiently small energies we are able to give an optimal Wegner estimate.
This means that its dependence on the volume on the system size and the energy interval length is linear.

\begin{thm}\label{t:optWE}
Let $H_\omega$ be a random Schr\"odinger operator as in Section \ref{s:DAmodel}.
There exists $\kappa\in (0,\infty)$ such that for any $q\in (0,1)$ there exists $C_W\in (0,\infty)$, such that,
for any  $L\in 2\NN+1$, $x \in \RR^d$, $E\in \RR$ and $\epsilon>0$
with $[E-\epsilon,E+\epsilon]\subset J:=[\lambda^{L,x}(0),\lambda^{L,x}(0) +q \kappa]$,
the following Wegner estimate holds
\begin{equation}
\label{e-WE2}
\EE\{\Tr [ \chi_{[E-\epsilon,E+\epsilon]}(H_{\omega, L,x}) ]\}
\le C_W \ s(\epsilon) \,\ |\Lambda_L|
\end{equation}
where $\lambda^{L,x}(0)=\inf \sigma(H_{0,L,x})$.
The constants $C_W$ and $\kappa$ depend only on
$C_-, C_+, \delta_-,\delta_+,M, \|V_0\|_\infty$
from Assumption \ref{hyp:Delone_Anderson} and on $q$.
\end{thm}

A Wegner estimate is by its very nature a statement about a finite box Hamiltonian.
However in typical applications it is important to know, how this box Hamiltonian
(actually, a sequence of box Hamiltonians on diverging scales), relates to the infinite volume system.
In particular, the minimum of the spectrum of  a (random) Schr\"odinger operator on the whole of $\RR^d$
is the reference energy when it comes to identifying appropriate intervals in which to expect Anderson localization. Using \eqref{eq:bottom-comparison} we obtain from the previous Theorem,

\begin{cor}
Let $H_{\omega,L,x}$ be a random Schr\"odinger operator as in Section \ref{s:DAmodel}, restricted to the box $\L_{L}(x)$ with Dirichlet boundary conditions.
There exists $\kappa\in (0,\infty)$ such that for any $q\in (0,1)$ there exists $C_W\in (0,\infty)$, such that,
for any  $L\in 2\NN+1$, $x \in \RR^d$, $E\in \RR$ and $\epsilon>0$
with $[E-\epsilon,E+\epsilon]\subset J:=[E_{\min},E_{\min} +q \kappa]$,
the following Wegner estimate holds
\begin{equation}
\label{e-WE2_Dirichlet}
\EE\{\Tr [ \chi_{[E-\epsilon,E+\epsilon]}(H_{\omega, L,x}) ]\}
\le C_W \ s(\epsilon) \,\ |\Lambda_L|
\end{equation}
where $E_{\min}=\inf \sigma(H_{0})$.
The constants $C_W$ and $\kappa$ depend only on
$C_-$, $ C_+$, $ \delta_-,\delta_+,M$, $ \|V_0\|_\infty$
from Assumption \ref{hyp:Delone_Anderson} and on $q$.
\end{cor}

\begin{proof}[Proof of Theorem \ref{t:optWE}]
We follow the strategy of the proof in \cite{CombesHK-07}.  To estimate $\EE\{\Tr [ \chi_{\Delta}(H_{\omega, L, x}) ]\}$, with $\Delta:=[E-\epsilon,E+\epsilon]$,
we decompose it with respect to the free spectral projector $\chi_{\tilde \Delta}(H_{0, L, x})$  of an interval $\tilde \Delta$,
such that $\Delta\subset \tilde \Delta$ and $d_{\Delta}=\dist(\Delta, \tilde\Delta^c)>0$, i.e.
\begin{equation}
\label{decomp} \Tr [ \chi_{\Delta}(H_{\omega, L, x})
]=\Tr [ \chi_{\Delta}(H_{\omega, L, x}) \chi_{\tilde\Delta}(H_{0, L, x}) ]+ \Tr [ \chi_{\Delta}(H_{\omega, L, x}) \chi_{\tilde\Delta^c}(H_{\omega, L, x}) ].
\end{equation}
The second term of the RHS in \eqref{decomp} can be estimated using Combes-Thomas type estimates as in \cite{CombesHK-07}.
To obtain a bound for the first term, uniform with respect to $x\in\RR^d$,
 the key is a positivity estimate analogous  to  \cite[Theorem 2.1]{CombesHK-07}.
In order to derive it, denote by $W$ the Delone-Anderson potential $V_\omega$ with all the random variables $\omega_j, j \in \Gamma_1, $ set equal to one,
and all random variables $\omega_j, j \in \Gamma_2, $ set equal to zero, i.e.,
\[
 W:= \sum_{j \in\Gamma_1}u_j
\]
Denote by $\lambda^{L,x}(0) := \inf \sigma(H_{0,L,x})$
the lowest eigenvalue of the restriction of $H_{0}$ to the box $\Lambda_L(x)$  with periodic or Dirichlet boundary conditions.
Now fix $q\in (0,1)$ and set $I=(-\infty,\lambda^{L,x}(0) +q \kappa]$.
Since $W$ is bounded, Lemma \ref{t:uncertainty} gives
\[
\chi_I(H_{0,L,x}) \, W \,  \chi_I(H_{0,L,x})  \geq (1-q)\kappa \, \chi_I(H_{0,L,x})
\]
Therefore, we have obtained the analog of \cite[Theorem 2.1]{CombesHK-07} for any interval $\Delta\subset I$.
The rest of the proof follows along the lines of \cite[Section 2]{CombesHK-07}.
\end{proof}
\smallskip

\section{Proof of Theorem \ref{t:scale_free-UCP}}
\label{s:proof-scale_free-UCP}

\subsection{Extension of the (sub)solution} \label{s:extension-solution}
We want to apply the quantitative unique continuation principle which requires among its geometric conditions
a certain security distance to the boundary of the set where the Schr\"odinger inequality is satisfied.
This is not true for the solution $\psi$ defined on the original cube $\LL$, therefore we will extend it to a larger set
in such a way that the extension $\tilde \psi$  still satisfies a Schr\"odinger inequality.
We have to discuss two different cases, namely periodic and Dirichlet boundary conditions on $\partial \Lambda_L$.
In the first case the extension is suggested by the very definition of periodic boundary conditions;
in the second we make use of the same construction as in Corollary A.2
in \cite{GerminetK}. Related constructions are used in Sobolev extension theorems,
cf.~for instance \cite{Agmon-10} or \cite{GilbargT-83},
and in Schwarz' method of image charges in electrostatics.

\subsection*{Extension in the case of periodic b.c.} We extend the potential $V$ as well as the function $\psi$,
defined on the box $\LL$, periodically to  $\tilde V, \tilde \psi \colon \RR^d\to\RR$.
Note that $\tilde V$ still takes values in $[-K_V,K_V]$ only. By the very definition of the domain $\cD(\Delta_{\L_L,\per})$
the extension $\tilde \psi$ is locally in the Sobolev space $W^{2,2}(\RR^d)$.
Moreover, the relation
\[
|\Delta \tilde \psi| \leq |\tilde V\tilde \psi|
\]
holds almost everywhere on $\RR^d$. Actually, we may consider the extension $\tilde \psi$ as a function on the
torus $\RR^d / (L\ZZ)^d$. We will make free use of this identification and denote the torus again by $\Lambda_L$, and the extension $\tilde \psi$ again by $\psi$ whenever there is no danger of confusion.

\subsection*{Extension in the case of Dirichlet b.c.}
\label{r:Dirichelt-extension}
The idea is to extend the potential $V$ by symmetric reflections
w.r.t. to hypersurfaces forming the boundaries of the cube $\Lambda_L$ and afterwards extend the function $\psi$
by antisymmetric reflections w.r.t. to hypersurfaces forming the boundaries of the cube $\Lambda_L$. Let us describe this more precisely.

For convenience we shift the coordinate system such that
\begin{align}
\L_L &= \{x \in \RR^d \mid -\frac{L}{2}\leq x_i \leq\frac{L}{2}  \text{ for all } i = 1, \ldots, d \}
\\
         &= \{y \in \RR^d \mid 0 \leq y_i \leq L \text{ for all } i = 1, \ldots, d \}
\intertext{and define a larger set of twice the size by}
R_L      &:= \{y \in \RR^d \mid -L \leq y_1 \leq L , \  0 \leq y_i \leq L \text{ for all } i = 2, \ldots, d  \}
\end{align}
If we define $\check V:R_L \to \RR$ by
\[
\check V(y_1, y_\perp) :=
\begin{cases}
 V(y_1, y_\perp), &\quad \text{ for }  y \in \L_L,
\\
0, &\quad \text{ for }  y \in R_L, y_1 =0,
\\
 V(-y_1, y_\perp), &\quad \text{ for }  y \in R_L, y_1<0
\end{cases}
\]
then $\check V$ still takes values in $[-K_V,K_V]$ only.
On the other hand we extend the function $\psi:\L_L \to \RR$ to the set $R_L$
by
\[
\hat \psi(y_1, y_\perp) :=
\begin{cases}
 \psi(y_1, y_\perp), &\quad \text{ for }  y \in \L_L,
\\
0, &\quad \text{ for }  y \in R_L, y_1 =0,
\\
 -\psi(-y_1, y_\perp), &\quad \text{ for }  y \in R_L, y_1<0
\end{cases}
\]
where $y_\perp=(y_2, \ldots,y_d)$.
It is well known that the Laplacian of this extension is still in $L^2(R_L)$,
and consequently $\hat\psi$ is in the domain of the Dirichlet Laplacian  on $R_L$,
cf. e.g. \cite{Agmon-10} or \cite{GilbargT-83}.
The relation
\begin{equation}
\label{eq:relationRL}
|\Delta \hat \psi| \leq |\check V\hat \psi|
 \end{equation}
still holds almost everywhere on $R_L$.

Now we successively extend both $\check V$ and $\hat \psi$
in the remaining $d-1$ directions and obtain functions defined on
\begin{align*}
         \{y \in \RR^d \mid -L \leq y_i \leq L \text{ for all } i = 1, \ldots, d,  \}
\end{align*}
i.e. on a cube twice the side of the original one. Finally we extend these two functions periodically w.r.t. the latice $(2L\ZZ)^d$
to functions defined on $\RR^d$ which satisfy \eqref{eq:relationRL} a.e. on $\RR^d$. Moreover $\hat \psi$ is in $W^{2,2}$ locally.

\begin{rem}
Note that if we restrict  the extension $\tilde\psi$ (or $\hat\psi$ respectively)
to a cube of side $2kL$ with $k \in \NN$ we obtain an $L^2$-eigenfunction of a box Schr\"odinger operator
with periodic  b.c. (or Dirichlet b.c).
\end{rem}

\subsection{Dominating and weak boxes}\label{s:dom-boxes}
We concentrate first on the case of periodic b.c. on $\partial \L_L$
and explain
below which modifications are necessary for Dirichlet boundary conditions.

For $L \in 2 \NN +1$, up to boundaries (sets of measure zero) $\L_L$ can be decomposed into closed unit boxes
\[
 \L_L= \bigcup_{k \in \tL_L} \L_1(k)
\]
where $\tL_L=\ZZ^d \cap [-L/2,L/2)^d$.

In some of our arguments we have to use the Euclidean norm in $\RR^d$
in others cubes, i.e.~ball in the sup-norm are natural. This is the reason why the factor $\sqrt{d}$ will appear often.

Fix $T= 62\ceil{\sqrt{d}}$, where $\ceil{x}$ stands for the least integer greater or
equal than $x$. The choice of $T$ will be apparent in formula \eqref{eq:choseT} when we want to make sure that the
geometric assumption of Corollary \ref{c:our_corollary} are satisfied.
We say that the site $k\in \tL_L$  is \emph{dominating} in the case of periodic b.c., if
\begin{equation}\label{eq:dombox}
 \int_{\L_1(k)} |\psi|^2 \geq  \frac{1}{2T^d} \int_{\L_T(k)} |\psi|^2
\end{equation}
and otherwise, \emph{weak}.

\begin{rem}
For eigenfunctions of one dimensional Schr\"odinger equations all sites $k\in \tL_L$  turn out to be dominating.
This is one of the reasons why the unique continuation estimates in one dimension obtained in the Diploma thesis \cite{Veselic-96}
and the papers
\cite{KirschV-02b,HelmV-07} are much simpler. The other reason is that Gronwall's lemma
gives a more effective quantitative unique continuation estimate than Theorem \ref{t:quantitative-UCP}.
\end{rem}

Let $W\subset\L_L$ be the union of unit boxes centered in weak  sites, and $D\subset\L_L$ the union of unit boxes centered in dominating sites. Note that $ \displaystyle \bigcup_{k \in\tL_L} \Lambda_T(k) $
considered modulo $(L\ZZ)^d$ gives an $T^d$-fold cover of the box (or torus) $\L_L$.
Since the extension $\tilde \psi$ is $(L\ZZ)^d$-periodic
this gives
\[
 \sum_{k\in \tL_L} \int_{\L_T(k) \text{ mod } (L\ZZ)^d} |\psi|^2
=\sum_{k\in \tL_L} \int_{\L_T(k)} |\tilde \psi|^2
=T^d \int_{\L_L}  |\psi|^2
\]
Thus
\begin{align} \label{eq:weak}
\int_W |\psi|^2 & =\sum_{\text{weak  sites}}  \int_{\L_1(k)} |\psi|^2
<  \frac{1}{2 T^d} \sum_{\text{weak  sites}}  \int_{\L_T(k)} |\tilde \psi|^2
\\ \nn
& \leq  \frac{1}{2T^d} \sum_{k\in \tL_L} \int_{\L_T(k)} |\tilde \psi|^2
=\frac{T^d}{2T^d} \int_{\L_L}  |\psi|^2 = \frac{1}{2} \int_{\L_L} |\psi|^2
\end{align}
We see that the weak  boxes contribute at most half of the total mass to the $L^2$ norm.
Since $D$ is the complement of $W$ in $\Lambda_L$
\begin{equation}
\label{eq:only_dominating}
2\int_D | \psi|^2  >   \int_{\L_L} |\psi|^2
\end{equation}
This means that it is sufficient to establish adequate unique continuation estimates for
dominating boxes.
\smallskip

\subsection*{Modification for Dirichlet boundary conditions}
\label{r:DirichletBC}
If $\psi$ enjoys Dirichlet b.c. on $\partial \LL$
then the associated extension $\hat \psi$ no longer satisfies
the equality
\[
 \sum_{k\in \tL_L} \int_{\L_T(k)} |\hat \psi|^2 =T^d \int_{\L_L}  |\psi|^2
\]
used in \eqref{eq:weak}. However, by the antisymmetry of $\hat \psi$ we have at least
\[
\int_{\L_T(k)} |\hat \psi|^2 \leq 2^d  \int_{\L_T(k)\cap \L_L}  | \psi|^2
\]
and thus
\[
 \sum_{k\in \tL_L} \int_{\L_T(k)} |\hat \psi|^2 \leq (2T)^d \int_{\L_L}  |\psi|^2 .
\]
For this reason we use in the case of Dirichlet b.c. on $\partial \LL$ a modified notion of dominating sites.
Here we call a site $k\in \tL$  \emph{(Dirichlet) dominating}, if
\begin{equation*}
 \int_{\L_1(k)} |\psi|^2 \geq  \frac{1}{2(2T)^d} \int_{\L_T(k)} |\hat \psi|^2
\end{equation*}
Then again  relation \eqref{eq:only_dominating} holds true.
\medskip

In the remainder of the proof we treat the case of periodic b.c. always first,
because it is easier.
Afterwards we discuss the changes necessary for Dirichlet b.c.

\subsection{A unique continuation principle for dominating boxes and near-neighbor sites}
\label{ss:rightNN}
For a dominating site $k\in\tL_L$ define its \emph{right near-neighbor} by
\begin{align}
 k^+ & =k+ (\ceil{\sqrt{d}}+1)\,{\mathbf e_1} \, \mod (L\ZZ)^d,
\text{ where } {\mathbf e_1}=(1,0,\ldots, 0)\in \RR^d
\end{align}
Thus  $\L_1(k^+)$ is a translation (on the torus) of $\L_1(k)$ in the positive direction along the first coordinate by $\ceil{\sqrt{d}}+1$.
Note that the map $\tL_L \ni k \mapsto k^+$ is injective.
By assumption, there exists a unique point $z_{k^+}$ of the Delone set $Z:=\{z_k\}_{k\in\Gamma}$
in $\L_1(k^+)$, and moreover, $B(z_{k^+},\delta)\subset \L_1(k^+)$.

Next we want to apply Corollary \ref{c:our_corollary} and for this purpose we have to check that the geometric assumptions are satisfied.
First note (see Figure \ref{A})
\[
\dist(\L_1(k), z_{k^+}) \geq |k-k^+| -\frac{1}{2} -\left(\frac{1}{2}-\delta\right)
\geq \ceil{\sqrt{d}}+\delta
\geq \diam \L_1(k)
\]
and
\[
\dist(\L_1(k), z_{k^+})
\leq |k-k^+| -\frac{1}{2} +\left(\frac{1}{2}-\delta\right)
\leq \ceil{\sqrt{d}}+1 -\delta
\leq \ceil{\sqrt{d}}+1
\]
Consequently $R$ will be in $\left[\ceil{\sqrt{d}}, 2\ceil{\sqrt{d}}\right]$.

\begin{figure}[ht]
\begin{center}
\includegraphics[width=9cm]{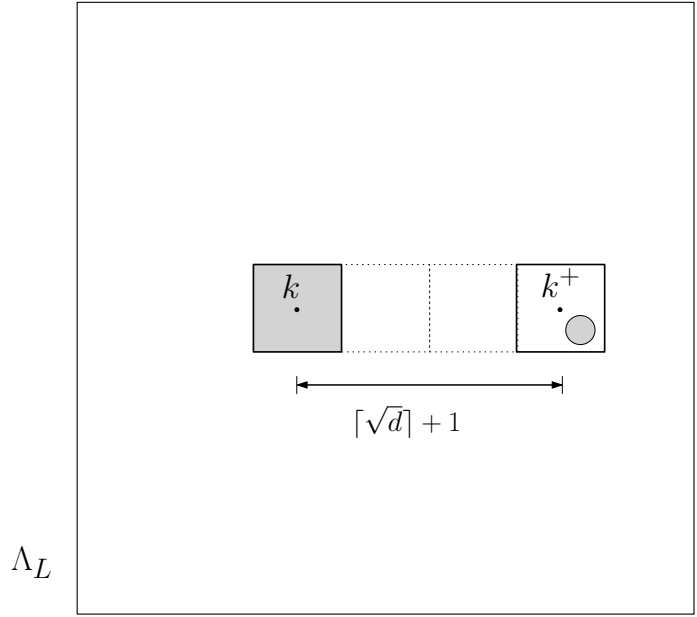}
\end{center}
\caption{Inside $\L_L$ in $d=2$, the dominating site $\L_1(k)$ and its right near-neighbor $k^+$.
The grey circle denotes the ball centered around the Delone point $z_{k^+}$.
The quantitative unique continuation principle
from Corollary \ref{c:our_corollary} relates the mass of $\psi$ on the grey areas.}
\label{A}
\end{figure}
\newpage
Moreover, for any  $y \in B(z_{k^+},14R)$
\begin{align}
\label{eq:choseT}
|k-y|
&\leq   \nonumber
|k-k^+| + |k^+ -z_{k^+}| + |z_{k^+}-y|
\leq \ceil{\sqrt{d}}+1+ \sqrt d +14 R
\\
&\leq 31 \ceil{\sqrt{d}} = T/2
\end{align}
Thus $B(z_{k^+},14R)\subset B(k, T/2) \subset \L_T(k)$.
Since $k \in \L_L$ is dominating, we have
\begin{equation}
\label{eq:beta-per}
\frac{\norm{\psi}_{\L_T(k)}^2}{\norm{\psi}_{\L_1(k)}^2} \leq  2T^{d}=:\beta =\beta_{\per}
\end{equation}
This means that we can apply Corollary \ref{c:our_corollary} with
\[
G= \L_T(k), \ \Theta = \L_1(k), \ x = z_{k^+}
\]
and the function $\psi\colon G \to \RR$.
It follows that for every dominating site $k$
\[
\norm{\psi}^2_{B(z_{k^+},\delta)}\geq C_{qUC} \norm{\psi}^2_{\L_1(k)}
\]
with the constant $C_{qUC}=C_{qUC}(d,K_{V},R,\delta,\beta_{\per})$
from Theorem \ref{t:scale_free-UCP}.
Taking the sum over all dominating sites $k \in \L_L$ we obtain
\begin{align*}
\label{eq:sum}
\sum_{k \in \tL_L \text{ dominating}} \norm{\psi}^2_{B(z_{k^+},\delta)}
&\geq C_{qUC} \sum_{k \in \tL_L \text{ dominating}} \norm{\psi}^2_{\L_1(k)}
\\
&\geq C_{qUC} \int_D \|\psi\|^2
\geq \frac{C_{qUC}}{2}\norm{\psi}_{\L_L}^2
\end{align*}
This ends the proof of Theorem \ref{t:scale_free-UCP} in the case of periodic b.c.
\medskip

\subsection*{Case of Dirichlet b.c.}
\label{r:DirichletBCend}
If the original (sub)solution $\psi$ has Dirichlet b.c. on $\partial \L_L$
we have to use the extension $\hat \psi$ described in Section \ref{r:Dirichelt-extension}
instead of the periodic one.
In this case we extend the set $Z_L:=Z \cap \L_L$ to a new set
\[
\hat Z_L:= \supp \widehat{{\bf 1}_{Z_L}}
\]
in the same way as the function $\hat \psi$. This $\hat Z_L$ is a Delone set and
is reflection symmetric {w.r.t.} to the hyperplane $\{x_1=\frac{L}{2}\}$.
A major difference to the proof of Theorem \ref{t:scale_free-UCP} for periodic boundary conditions
is that in the case of Dirichlet b.c. for dominating sites $k \in\tL_L$ we define
the right near-neighbor as
\[
k^+:=
k+ (\ceil{\sqrt{d}}+1)\,{\mathbf e_1},  \text{ where } {\mathbf e_1}=(1,0,\ldots, 0)\in \RR^d
\]
without taking $\!\!\!\mod (L\ZZ)^d$.
Thus there will be two cases: either the right near neighbor $k^+$ is still in $\tL_L$
(the set of such $k\in \tL_L$ will be denoted by $A$), or,
for $k$ with distance less than $\lceil \sqrt d\rceil +1$ to the right boundary of $\L_L$,
$k^+$ will be outside $\L_L$.
The set of such $k\in \tL_L$ will be denoted by $B$.
Note that even for $k\in B$ we still have:
\[
L/2 \leq k^+_1 \leq (L/2)  +\lceil \sqrt d\rceil +1 ,
\quad |k_j| \leq L/2 \text{ for, } j =2, \ldots , d
\]
There  is a unique Delone point $z_{k^+}\in \hat Z_L \cap \L_1(k^+)$
(which is in general not in the original Delone set $Z$).

If $L \geq \lceil \sqrt d\rceil +1 $ set
\begin{equation}
\label{eq:kplusminus}
k^{+-} := {k^+} -2(k^+_1- L/2) {\mathbf e_1}
= \left(  -k^+_1+ L, k^+_2,\ldots,k^+_d \right)
\end{equation}
This $k^{+-}$ is the mirror image of $k^+$ w.r.t.~to the hyperplane $\{x_1=\frac{L}{2}\}$,
cf.~Figure \ref{B}. For $k \in A$ we set $ k^{+-} := k^+$. Then, by construction, $k^{+-}$ is inside $\tL_L$.
There is a unique $z_{k^{+-}}\in Z\cap \L_1(k^{+-})$ and by the reflection symmetry  of $|\hat \psi|$
we have
\[
\|\psi\|_{B(z_{k^{+-}},\delta ) }
=
\|\hat \psi\|_{B(z_{k^+},\delta ) }
\]
Thus the application of Corollary \ref{c:our_corollary} as above
with $\beta=\beta_{\text{Dir}} :=2 (2T)^d$ yields
\begin{align*}
\sum_{k \in \tL_L \text{ dominating}} \norm{\psi}^2_{B(z_{k^{+-}},\delta)}
=
\sum_{k \in \tL_L \text{ dominating}} \norm{\hat\psi}^2_{B(z_{k^+},\delta)}
\\
\geq C_{qUC} \sum_{k \in \tL_L \text{ dominating}} \norm{\psi}^2_{\L_1(k)}
= C_{qUC} \int_D \|\psi\|^2
\geq \frac{C_{qUC}}{2}\norm{\psi}_{\L_L}^2
\end{align*}
Note that the two restricted maps
\begin{align*}
A \ni & k \mapsto k^+ \mapsto  k^{+-}, \\
B \ni & k \mapsto k^+ \mapsto  k^{+-}
\end{align*}
are both injections. Thus in the sum
\[
\sum_{k \in \tL_L \text{ dominating}} \norm{\psi}^2_{B(z_{k^{+-}},\delta)}
\]
each site $k^{+-}$ may appear at most twice. Consequently
\begin{align}
\label{eq:all-dominating}
\norm{\psi}_{\L_L}^2
\leq
\frac{2}{C_{qUC}}
\sum_{k \in \tL_L \text{ dominating}} \norm{\psi}^2_{B(z_{k^{+-}},\delta)}
\leq
\frac{4}{C_{qUC}}
\sum_{k \in \tL_L} \norm{\psi}^2_{B(z_{k},\delta)}
\end{align}
If $L < \lceil \sqrt d\rceil +1 $, estimate \eqref{eq:all-dominating} again holds,
although the assignment $k^+\mapsto k^{+-}$ has to be chose somewhat differently than
\eqref{eq:kplusminus}.

Note that the constant $C_{qUC}$ in the case of Dirichlet boundary conditions
differs from the one in case of periodic b.c. through the parameter
$\beta_{\text{Dir}}= 2 (2T)^d= 2^d \beta_{\per}$.

\begin{figure}[ht]
\begin{center}
\includegraphics[width=10cm]{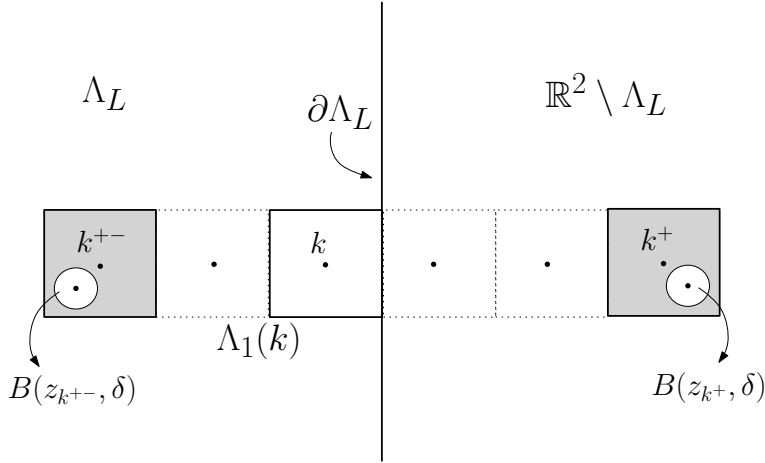}
\end{center}
\caption{A box near the boundary of $\L_L$.  In grey: its right-near neighbor and the corresponding mirror image in $\L_L$.}
\label{B}
\end{figure}

\appendix

\section{Proof of quantitative unique continuation estimate}

\subsection{Carleman and Cacciopoli estimate}
The crucial tool in the proof of Theorem \ref{t:quantitative-UCP}
is a Carleman estimate as derived in \cite{Hoermander-83,EscauriazaV-03}.
Its proof can be found in \cite{BourgainK-05} as well.
For us a scaled version from the Appendix of \cite{GerminetK}, which we state below, will be convenient.
Let $\rho >0$ and set
\begin{align}
\varphi &\colon [0,\infty) \to [0,\infty),
&\varphi(s)&:=  s\cdot\exp \left(-\int_0^s \frac {1 - \e^{-t}} t \, d t\right)
\\
w_\rho&\colon \RR^d\to [0,\infty[,
&w_\rho(x)&:=\varphi(|x|/\rho)
\end{align}
Then $\varphi$ is a strictly increasing continuous function, on $(0, \infty)$ even smooth,
and
\begin{equation}
\label{eq:carleman_condition}
\frac{|x|}{3 \rho}  \leq w_\rho(x) \le \frac{|x|}{ \rho}
\quad\text{ for all } x \in B(0,\rho)
\end{equation}

While the usual version of Carleman's estimate is valid for smooth functions,
a regularisation procedure, cf. e.g.\, \cite[Theorem 1.6.1]{Ziemer-89}, allows one to extend it
to functions in the second Sobolev space. Here is \cite[Lemma A.5]{GerminetK}:
\begin{lem}\label{l:Carleman}
Let $\rho$ and $w_\rho$ be as above.
There are constants $C_2,C_3 \in [1,\infty)$, depending only on the dimension $d$, such that for all
  $\alpha \ge C_2$ and  all real-valued $f\in W^{2,2}( B(0,\rho))$ with compact support in $B(0,\rho)\setminus \{0\}$
\begin{equation}
\label{eq:carleman_estimate}
\alpha^3 \int_{\RR^d}w_\rho^{-1-2\alpha}(x) f^2(x)  \, d x
\leq
 C_3 \rho^4 \int_{\RR^d}  w_\rho^{2-2\alpha}(x) (\Delta f(x))^2  \, d x
\end{equation}
holds.
\end{lem}

The second tool which will be used several times are interior gradient estimates
and the Cacciopoli inequality. Due to the spherical symmetry of the integration domains
the constants take on a very simple form.

\begin{lem}[Interior gradient estimate]
\label{l:interior}
Let $0\leq a< c < \infty$, $G \subset\RR^d$ open,
$S:= B(c)\setminus B(a)=\{ x \mid a \leq |x| < c\}$,
$V\colon G\to \RR$ bounded and measurable,
and $\psi \in W^{1,2}(G)$ with $\Delta \psi \in L^2(G)$.
\begin{enumerate}[(a)]
 \item
If $a>0$ and $b \in (0,a)$ set
$S^+ := \overline{B(c+b)}\setminus B(a-b)=\{ x \in \RR^d\mid a-b \leq |x| \leq c+b\}$
and assume $S^+ \subset G$. Then
\[
\|\nabla\psi\|_S^2
\leq
\|(-\Delta+V) \psi \|_{S^+}^2 +
\left(1 +  \frac{9}{ b^2} \right)\|\psi \|_{S^+}^2
+2\int_{S^+} V_- |\psi|^2
\]
\item
If  $b \in (0,\infty)$ and $S^+:=\overline{B(c+b)}=\{ x \mid |x| \leq c+b\}\subset G$, then
\[
\|\nabla\psi\|_{B(c)}^2
\leq
\|(-\Delta+V) \psi \|_{S^+}^2 +
\left(1 +  \frac{9}{ b^2} \right)\|\psi \|_{S^+}^2
+2\int_{S^+} V_- |\psi|^2
\]
\end{enumerate}
\end{lem}

\begin{proof}
For $\eta \in C_0^1(G)$
\[
\|\eta\nabla\psi\|^2
\leq
\|\eta (-\Delta + V) \psi \|^2 + \|\eta  \psi \|^2
+4 \|\psi \nabla \eta \|^2
+2\langle \eta\psi, V_- \eta \psi\rangle
\]
In the case $ 0 \leq \eta \leq 1$, $\supp \eta = S^+$, $\eta|_S\equiv 1$ we obtain
\begin{equation*}
\|\nabla\psi\|_S^2
\leq
\|\eta\nabla\psi\|^2
\leq
\|(-\Delta + V) \psi \|_{S^+}^2 + \|\psi \|_{S^+}^2
+4 \|\psi \nabla \eta \|^2
+2\int_{S^+} V_- |\psi|^2
\end{equation*}
Now we choose $\eta$ spherically symmetric such that the gradient is given by the radial derivative
and insert the appropriate sets $S$ and $S^+$.
\end{proof}
An immediate consequence is
\begin{lem}[Cacciopoli inequality]
\label{l:Cacciopoli}
Let $0\leq a< c < \infty$, $G \subset\RR^d$ open,
$S:= \overline{B(c)}\setminus B(a)=\{ x \mid a \leq |x| \leq c\}$,
$V\colon G\to \RR$ bounded and measurable,
and $\psi \in W^{1,2}(G)$ with $\Delta \psi \in L^2(G)$.
Assume that $\psi$ satisfies \eqref{eq:subsolution}.
Then
\begin{enumerate}[(a)]
 \item
Let  $a>0$ and $b \in (0,a)$. Assume $\overline{B(c+b)}\setminus B(a-b)=\{ x \mid a-b \leq |x| \leq c+b\}\subset G$.
Then
\[
\int_{B(c)\setminus B(a)} |\nabla\psi |^2
\leq
\left(1 +  \frac{9}{ b^2} + \|V\|_\infty^2\right) \int_{B(c+b)\setminus B(a-b)} |\psi |^2
\]
\item
Let $a\geq 0$ and $b \in (0,\infty)$. Assume $\overline{B(c+b)}=\{ x \mid |x| \leq c+b\}\subset G$.
Then
\[
\int_{B(c)\setminus B(a)} |\nabla\psi |^2
\leq
\left(1 +  \frac{9}{ b^2} + \|V\|_\infty^2\right) \int_{B(c+b)} |\psi |^2
\]
\end{enumerate}
\end{lem}

For the prefactors appearing on the  RHS we will use the symbol $C_{V,b}$.

\subsection{Carleman estimate implies Theorem \ref{t:quantitative-UCP} }
To achieve zero boundary values we will use a function $\eta \colon \RR^d \to [0,1]$, $\eta\in  C^\infty_0$ depending only on $|x|$ and satisfying
\begin{align}
\label{eq:eta}
 \eta             &=0     &\text{ on } {B\left(\frac{\delta}{8}\right)} \cup B(12R+D_0)^c
\\
\nonumber
 \eta             &=1     &\text{ on } {B(12R)} \setminus B\left(\frac{\delta}{4}\right)
\\
\nonumber
   \max\left( \|\nabla\eta\|_\infty, \|\Delta\eta\|_\infty \right)
& \leq  \left(\frac{K_{\Delta}}{\delta}\right) ^2 & \text{ on } {B\left(\frac{\delta}{4}\right)}
 \setminus B\left(\frac{\delta}{8}\right)
\\
\nonumber
   \max\left( \|\nabla\eta\|_\infty, \|\Delta\eta\|_\infty \right)
& \leq \left( \frac{K_{\Delta}}{D_0}\right)^2 & \text{ on } {B(12R+D_0)} \setminus B(12R)
\end{align}
Here $K_\Delta\in (0,\infty)$ is an absolute constant.
We set $\rho=24R$  and apply the Carleman estimate in Lemma \ref{l:Carleman} to the function $f=\eta\psi$
\begin{align*}
\alpha^3 \int w_\rho^{-1-2\alpha}(\eta\psi)^2
& \leq C_3\rho^4\int w_\rho^{2-2\alpha}\left( 4(\Delta \eta)^2\psi^2+16 |\nabla\eta|^2|\nabla\psi|^2+4\eta^2(\Delta\psi)^2  \right)
\end{align*}
where we used $(a+b+c)^2\leq 4(a^2+b^2+c^2)$.
Hence
\begin{equation}
 \label{eq:AB}
\frac{\alpha^3}{4C_3\rho^4} \int w_\rho^{-1-2\alpha}(\eta\psi)^2
\leq
 A + B
\end{equation}
where
\begin{align*} A& :=\int w_\rho^{2-2\alpha}\eta^2(\Delta\psi)^2\\
B & := \int_{\supp\nabla\eta} w_\rho^{2-2\alpha}\left( (\Delta \eta)^2\psi^2 +4|\nabla\eta|^2|\nabla\psi|^2 \right)
\end{align*}
\smallskip

First we derive an upper  {bound on A}.
Since $w_\rho\leq 1$ on $B(0,\rho)$ and
$\psi$ satisfies the  differential inequality
$|\Delta \psi(x)|\leq |V\psi(x)| $:
\begin{equation}
\begin{split}
 A=\int w_\rho^{2-2a} \eta^2 (\Delta \psi)^2
& \leq
 \int w_\rho^{2-2a} \eta^2 (V\psi)^2 \\
& \leq
K_{V}^2 \int w_\rho^{-1-2a} \eta^2 \psi^2
 \end{split}
\end{equation}
where  $K_{V}:=\norm{V}_\infty $.
Inserting this in \eqref{eq:AB} yields,
\begin{align}
\label{eq:subsume_V}
\left(\frac{\alpha^3}{4C_3\rho^4} - K_{V}^2 \right)\int w_\rho^{-1-2a} \eta^2 \psi^2 &\leq B
\end{align}
\smallskip

We choose now $\alpha$ so that it yields and efficient lower bound on the LHS of \eqref{eq:subsume_V}
and satisfies properties needed for \eqref{eq:D0subs} to hold.

Let $\alpha$ be such that
\begin{subequations}
\label{eq:choicealpha}
\begin{equation}
\label{eq:choicealpha1}
\alpha \geq C_2 >1, \quad \text{ }
\end{equation}
where $C_2$ is the constant from the Carleman estimate,
\begin{equation}
\label{eq:choicealpha2}
\alpha \geq  \sqrt[3]{24^5 C_3 K_V^2 R^4}
\end{equation}
\begin{equation}
\label{eq:choicealpha3}
\begin{array}{cll}
&\alpha &\geq \frac{1}{2} \ln\left( \left(\frac{24RK_{\Delta}}{D_0}\right)^4 \frac{C_3(1+4C_{V,D_0})}{45}\beta\right)
 \end{array}
\end{equation}
 \end{subequations}
Thus \eqref{eq:subsume_V}  and (\ref{eq:choicealpha2}) ensure
\begin{equation}
\label{eq:AB2}
\frac{5}{6}\frac{\alpha^3}{4C_3 \rho^4} \int w_\rho^{-1-2a} \eta^2 \psi^2
\leq B
\end{equation}
\smallskip

We derive now {a lower bound on the LHS of \eqref{eq:AB2}}.
Since $-1-2\alpha<0$ and $\frac{|x|}{3\rho}\leq w_\rho\leq \frac{|x|}{\rho}$, we have

\begin{equation}\label{eq:wbound}
\frac{(3\rho)^{1+2\alpha}}{|x|^{1+2\alpha}}\geq w_\rho^{-1-2\alpha}\geq \frac{\rho^{1+2\alpha}}{|x|^{1+2\alpha}}
\end{equation}
Since $4R \geq \delta$ and  $\Theta\subset B(0,2R)\setminus B(0,R)$ we have
$\eta\equiv 1$ on $\Theta$.
Therefore, by \eqref{eq:wbound} and recalling $\rho=24R$, we get
\begin{equation}
\int w_\rho^{-1-2\alpha}(\eta\psi)^2 \geq \int_{\Theta} w_\rho^{-1-2\alpha}\psi^2\geq \left(\frac{\rho}{2R}\right)^{1+2\alpha} \int_\Theta\psi^2
=
12^{1+2\alpha} \norm{\psi}_\Theta^2
\end{equation}
So \eqref{eq:AB2} implies
\begin{equation}
\label{eq:AB3}
\frac{5}{2} \cdot\frac{\alpha^3}{C_3 \rho^4}12^{2\alpha} \norm{\psi}_\Theta^2=\frac 5 6 \cdot\frac{\alpha^3}{4C_3 \rho^4}(12)^{1+2\alpha} \norm{\psi}_\Theta^2 \leq B
\end{equation}
\smallskip

We turn now to an upper {bound on $B$}.
We first use $w_\rho^{2-2\alpha}\leq \left(\frac{3\rho}{|x|}\right)^{2\alpha-2}$ to obtain
\begin{align}
B &
\leq \int_{\supp\nabla\eta} \left(\frac{3\rho}{|x|}\right)^{2\alpha-2}\left( (\Delta \eta)^2\psi^2 +4|\nabla\eta|^2|\nabla\psi|^2 \right)
\\&
= \int_{\supp\nabla\eta} \left(\frac{72R}{|x|}\right)^{2\alpha-2}\left( (\Delta \eta)^2\psi^2 +4|\nabla\eta|^2|\nabla\psi|^2 \right)
\end{align}
Now we split the integral according to the two components of
\begin{equation*}
\supp\nabla\eta
\subset \left\{ \frac{\delta}{8} \leq\abs{x}\leq \frac{\delta}{4}\right\}
\cup \left\{ 12R \leq\abs{x}\leq 12R+D_0\right\}
\end{equation*}
and obtain, using the upper bound on the derivatives of $\eta$ and   the monotonicity of $|x|^{-1}$.
\begin{align}
B
& \leq   \int_{\frac{\delta}{8}\leq |x| \leq \frac{\delta}{4}}
{\left(\frac{8\cdot 72R}{\delta}\right)^{2\alpha -2}}
\left[{\left(\frac{K_{\Delta}}{\delta}\right)^4 \psi^2 + 4 \left(\frac{K_{\Delta}}{\delta}\right)^4 |\nabla\psi|^2}\right]
\\
&
+ \int_{12R\leq |x| \leq 12R+D_0}
{ 6^{2\alpha -2}}
{\left(\frac{K_{\Delta}}{D_0} \right)^4 (\psi^2 + 4|\nabla\psi|^2)}
\end{align}
In the next step we want to use the Cacciopoli inequality in Lemma \ref{l:Cacciopoli} and enlarge therefore the set
$\left\{\frac{\delta}{8}\leq |x| \leq\frac{\delta}{4}\right\}$ to
$\left\{ |x| \leq \frac{3\delta}{4}\right\} \subset B(0,\delta) $
and $\{12R\leq|x|\leq 12R+D_0\}$ to
$\{|x|\leq 12R+2D_0  \} \subset G$.
Thus we obtain
\begin{align}
B &
\leq \left(\frac{12\cdot 48R}{\delta}\right)^{2\alpha -2} \left(\frac{K_{\Delta}}{\delta}\right)^4 (1+4C_{V,\frac{\delta}{2}}) \|\psi\|_{B(0,\delta)}^2
\\
& \quad   \label{eq:Bbound}
+ 6^{2\alpha -2}\left(\frac{K_{\Delta}}{D_0}\right)^4 (1+ 4C_{V,D_0}) ||\psi||_G^2
\end{align}
Now we want to subsume the  term which depends on $D_0$ into the lower bound.
More precisely,  we want to show that
\begin{align}
\label{eq:D0subs}
\frac{5}{4} \cdot\frac{\alpha^3}{C_3 \rho^4}12^{2\alpha} \norm{\psi}_\Theta^2 & \geq  \left(\frac{K_{\Delta}}{D_0}\right)^4   6^{2\alpha-2}(1+ 4C_{V, D_0}) \norm{\psi}_G^2
\end{align}
Assumption (\ref{eq:choicealpha3}) implies
\begin{align*}
4^{\alpha} &\geq \frac{(24R)^4}{45} \left(\frac{K_{\Delta}}{D_0}\right)^4   C_3(1+ 4C_{V, D_0}) \beta
\intertext{which simplifies to}
\frac{5}{4} \frac{\alpha^3}{C_3 \rho^4} 2^{2\alpha}
& \geq \left(\frac{K_{\Delta}}{D_0}\right)^4  6^{-2}(1+ 4C_{V, D_0}) \beta
\text{ i.e. }(\ref{eq:D0subs})
\end{align*}
Hence  the Assumptions \eqref{eq:choicealpha} together with \eqref{eq:AB3}
and inequality \eqref{eq:Bbound}  imply
\begin{align}
\frac{5}{4}\frac{\alpha^3}{C_3 \rho^4} \norm{\psi}_\Theta^2
\leq
\frac{1}{12^2} \left(\frac{48R}{\delta}\right)^{2\alpha -2} \left(\frac{K_\Delta} {\delta}\right)^4
& (1+4C_{V,\frac{\delta}{2}}) \norm{\psi}^2_{B(0,\delta)}
\end{align}
We use the bound (\ref{eq:choicealpha1}) and obtain:
\begin{equation}
\label{eq:boundB1}
\begin{split}
\norm{\psi}^2_{B(0,\delta)}
&\geq
\frac{5}{4} \frac{C_2^3}{C_3} K_\Delta^{-4} R^{-2} \delta^2 (1+4C_{V,\frac{\delta}{2}})^{-1} \left( \frac{\delta}{48R} \right)^{2\alpha}
 \norm{\psi}^2_\Theta
 \end{split}
\end{equation}
Now we insert in the estimate the bounds from the Cacciopoli inequality. Using the abbreviation  $ b_1 =\min(b,1)$ we conlude
\begin{align}
1+4C_{V,b}
\leq
\frac{41+4\norm{V}_\infty^2}{ b_1^2}
\end{align}
Thus we obtain as a lower bound
\begin{align}
\label{eq:nr1}
\norm{\psi}^2_{B(0,\delta)}
&\geq \frac{5}{16} \frac{C_2^3}{C_3 K_\Delta^4} \frac{\delta^4}{R^2} \left(\frac{\delta}{48R}\right)^{2\alpha} \frac{1}{41} \frac{1}{1+\norm{V}_\infty^2} \norm{\psi}_\Theta^2.
\end{align}
Note that by the Cacciopoli inequality \eqref{eq:choicealpha3} is satisfied if
\begin{align}
 \label{eq:Ungl3}
2\alpha \geq \ln\left(\left(\frac{24RK_\Delta}{D_0}\right)^4 C_3 \frac{1+\norm{V}_\infty^2}{D_1^2} \beta\right)
\end{align}
where we set $D_1=\min(D_0,1)$.

\subsection{Special cases treated in Corollary \ref{c:our_corollary}}

We specialize to the case $\sqrt{d} \leq R\leq D_0$.
Then (\ref{eq:Ungl3}) is satisfied if
\begin{align*}
2\alpha &\geq
\ln\left(	c_4 (1+\norm{V}_\infty^2) \beta\right),
\quad
{c_4:=  (24K_\Delta)^4 C_3}
\end{align*}
In particular, all three conditions \eqref{eq:choicealpha} are satisfied if
\[
2\alpha
\geq
2C_2
+ 2(24^5 C_3)^{1/3} \, ( K_V^2 R^4)^{1/3}
+ \ln\left(c_4 (1+\norm{V}_\infty^2)\right) + \ln \beta
\]
Under this condition we have
\begin{multline*}
\norm{\psi}^2_{B(0,\delta)}
\geq
c_0 \frac{\delta^4}{R^2} \left(\frac{\delta}{48R}\right)^{2\alpha}  \frac{1}{1+\norm{V}_\infty^2} \norm{\psi}_\Theta^2
\\
\geq
c_0 \left(\frac{\delta}{48R}\right)^{4+2C_2
+ 2(24^5 C_3)^{1/3} \, ( K_V^2 R^4)^{1/3}
+ \ln\left(c_4 (1+\norm{V}_\infty^2)\right) + \ln \beta
}  \frac{1}{1+\norm{V}_\infty^2} \norm{\psi}_\Theta^2
\end{multline*}
where $c_0:= \frac{5}{16} \frac{C_2^3}{C_3 K_\Delta^4} \frac{1}{41} $.
Now this means that for a sufficiently large ${C} \in (1,\infty)$ depending only on the dimension
we have:
\begin{align}
\label{eq:clear_CqUC}
\norm{\psi}^2_{B(0,\delta)}
&\geq
\left(\frac{\delta}{{C} R}\right)^{{C}+ {C} K_V^{2/3} R^{4/3}
+ \ln \beta}  \norm{\psi}_\Theta^2
\end{align}
\medskip

Now we restrict our choice of parameters even further and assume
\begin{align*}
\sqrt{d}\leq R \leq D_0, \ R \leq 2 \lceil\sqrt{d}\rceil
\end{align*}
Then \eqref{eq:clear_CqUC} implies
that there exists a constant ${C} \in (1,\infty)$ depending only on the dimension $d$  such that
\begin{align}
\norm{\psi}^2_{B(0,\delta)}
&\geq
\left(\frac{\delta}{{C}}\right)^{{C} + {C} K_V^{2/3} + \ln \beta}
\norm{\psi}^2_\Theta
\end{align}

\bigskip

\section{Control of local fluctuations inside dominating boxes}

In the proof of Theorem \ref{t:scale_free-UCP} we derived a lower bound on the total $L^2$-mass 
$\|\psi\|_{\L_1(k)}^2$ inside a box $\L_1(k)$ of a dominating site $k$, 
in comparison with the mass near the corresponding right near neighbor of $k$, cf. \S \ref{ss:rightNN}.
Here we present a more detailed statement about mass fluctuations on small scales inside a dominating box $\L_1(k)$. 
The dependence on the parameters is quite involved and the statement is not used in the main body of the paper.
Nevertheless, the control of mass fluctuations on small scales may be useful in other contexts.

\begin{thm}
\label{t:local-fluctuations}
Let $\delta \in(0,1/20]$ be fixed, $V$, $\L_L$ and $\psi$ as in Theorem \ref{t:scale_free-UCP}
$k \in\tL_L$ a dominating site, as defined in Section \ref{s:dom-boxes} .
Then there exists a constant $C_{lf}\in(0,\infty)$ such that for any $B(x,\delta) \subset \L_1(k)$ we have
\[
\int_{B(x,\delta)} |\psi|^2
\geq
C_{lf} \int_{\L_1(k)} |\psi|^2
\]
The constant $C_{lf}$ is given explicitly in \eqref{appendix-eq:second}.
\end{thm}

In the proof we will use the following Corollary of Theorem \ref{t:quantitative-UCP}.

\begin{cor}
\label{appendix-c:our_corollary}
Let $K_V, R, \beta\in [0, \infty)$,  $\delta\in (0,1]$.
Let $G\subset  \RR^d$  open, $x\in G$, $\Theta\subset G$ measurable,
satisfy the geometric conditions
 \begin{equation}
 \label{appendix-eq:geometric_conditions}
\diam \Theta \leq  R =\dist (x,\Theta),
\
0< \frac{\delta }{4} \leq \frac{1}{10}\leq R \leq \sqrt{d}, \text{ and }   \ B(x, 14R)\subset G,
 \end{equation}
and $V\colon G \to [-K_V,K_V]$ measurable, $\psi\in W^{2,2}(G)$ real-valued, satisfy
\begin{equation}
|\Delta \psi| \leq |V\psi| \quad \text{a.e. on }   G
\quad  \text{ and } \quad \frac{\norm{\psi}_G^2}{\norm{\psi}_\Theta^2} \leq \beta
\end{equation}
Then there exists a constant ${C}={C}(d)\in (1,\infty)$ depending only on the dimension, such that
\begin{align}
\label{appendix-eq:our_corollary}
\norm{\psi}^2_{B(x,\delta)}
&\geq
C_{qUC}
\norm{\psi}_\Theta^2
\ \text { where } \
C_{qUC}:=
\left(\frac{\delta}{{C}}\right)^{{C} + {C} K_V^{2/3} + \ln \beta}
\end{align}
\end{cor}

\begin{proof}[Proof of Theorem \ref{t:local-fluctuations}]
For the proof we consider the extensions of  $\psi$ to periodic and Dirichlet b.c. defined in Section \ref{s:extension-solution}, together with the respective definitions of dominant boxes from Section \ref{s:dom-boxes}. 
In what follows we focus on the case of periodic b.c.
As before  we consider the periodic extension of $\psi$ as a function on the torus 
and denote it again by $\psi$.
The case of Dirichlet b.c. is treated completely analogously, 
up to the fact that an additional factor $2^d$ appears in the constants $\beta$ and $\tilde\beta$ below.

\subsection{Maximal box in a  dominating box}
Let us recall that for $T=30\ceil{\sqrt{d}}$, a dominating box $\L_1(k)$ in the case of periodic b.c. satisfies
\be\label{appendix-eq:dombox}
 \int_{\L_1(k)} |\psi|^2 \geq  \frac{1}{2 T^d} \int_{\L_T(k)} |\psi|^2.
\ee
Note that the choice of $T$ differs from the one in Section \ref{s:proof-scale_free-UCP}.2. This will become apparent later on, when setting the range of values for the distance parameter $R$, taking into consideration the geometric conditions \eqref{appendix-eq:geometric_conditions}.

We split $\L_1(k)$ into $(10\ceil{\sqrt{d}})^d$ (almost) disjoint boxes of sidelength $1/(10\ceil{\sqrt{d}})$.
Then there exists at least one
\begin{equation*}
x_k\in\left\{\left(\frac{1}{10\ceil{\sqrt{d}}}\ZZ\right)^d+\left( \frac{1}{20\ceil{\sqrt{d}}},\frac{1}{20\ceil{\sqrt{d}}},\cdots, \frac{1}{20\ceil{\sqrt{d}}} \right) \right\}\cap\L_1(k)
\end{equation*}
such that we have
\begin{equation}\label{appendix-maxbox}
 \int_{\L_{\frac{1}{10\ceil{\sqrt{d}}}}(x_k)} |\psi|^2 \geq  \frac{1}{(10\ceil{\sqrt{d}})^d} \int_{\L_1(k)} | \psi|^2
\end{equation}
Fix such a $x_k$, and call $\Theta_k:=\L_{\frac{1}{10\ceil{\sqrt{d}}}}(x_k)$ the \emph{maximal box}.
In particular, by the definition of a dominating box, cf. \eqref{appendix-eq:dombox},
\[
 \int_{\Theta_k} |\psi|^2 \geq \frac{1}{2(10\ceil{\sqrt{d}}T)^d} \int_{\L_T(k)} |\psi|^2
\]
and hence
\begin{equation}\label{appendix-quotient}
\frac{\norm{\psi}_{\L_T(k)}^2}{ \norm{\psi}_{\Theta_k}^2 } \leq 2(10\ceil{\sqrt{d}} T)^{d}:=\beta
\end{equation}

\subsection{UCP inside a dominating box}

We want to show that for a fixed $\delta \in(0,\frac{1}{20}]$, there exists a positive constant $C_{lf}$ such that for any ball $B(x,\delta)$
contained in a dominating $ \L_1(k)$, we have
\begin{equation}\label{appendix-aim}
    \int_{B(x,\delta)} |\psi|^2 \geq C_{lf}  \int_{\L_1(k)} |\psi|^2
 \end{equation}

\smallskip

The strategy to prove (\ref{appendix-aim}) is the following:
\begin{enumerate}[(i)]
 \item Take a ball $B(x,\delta)$ contained in a belt $A$ of inner radius $3/10$,
around the center $x_k$ of the maximal box $\Theta_k$ (so $x$ is far enough from $\Theta_k$) and apply Corollary \ref{appendix-c:our_corollary}.  We will obtain the desired estimate for all balls except for those in a neighborhood of $\Theta_k$.
 \item
To cover balls that are close to $\Theta_k$ not considered in (i), fix a box in $A$ as a new ``maximal box'',
which by the previous step satisfies an estimate of the type (\ref{appendix-maxbox}), and
is sufficiently far away from balls close to $x_k$.  Repeat Corollary \ref{appendix-c:our_corollary} to get (\ref{appendix-aim}) .
\end{enumerate}

\subsection{First case}
\label{appendix-first-est}

We want to apply Corollary \ref{appendix-c:our_corollary} to $\psi$
with $G=\L_T(k)$ and $\Theta=\Theta_k$ and some $x$ in a dominating box $\L_1(k)$
such that $B(x,\delta)\subset \L_1(k)$.

Consider the belt
\[
 A(x_k):= \L_1(k)\setminus \L_{\frac{3}{10}}(x_k) ,
\]
fix $0<\tilde\delta\leq 1/20$, and take $x\in A(x_k)$,  such that $B(x,\tilde\delta)\subset \L_1(k)$.
Thus
\[
B(x,\tilde\delta)
\subset \L_1(k)\setminus \L_{\frac{1}{10\ceil{\sqrt{d}}}}(x_k)
\quad  \text{ and } \quad
R:=\dist (x,\Theta) \geq \frac{1}{10}
\]
Hence $R \in[\frac{1}{10},\sqrt{d}] $,  and   since $ T=30\ceil{\sqrt{d}}$,
the geometric conditions \eqref{appendix-eq:geometric_conditions}
are satisfied.
 We obtain:
\begin{equation}\label{appendix-UCP1}
\norm{\psi}_{B(x,\tilde\delta)}^2 \geq  C_{qUC}(\tilde\delta, \beta) \norm{\psi}_{\Theta_k}^2
\end{equation}
where $C_{qUC}(\tilde\delta, \beta) =C_{qUC}(d,K_V, D_0, R,\tilde \delta, \beta) $
is defined in Corollary \ref{appendix-c:our_corollary}. In this proof we write only the dependence on the parameters $\tilde\delta$ and $\beta$
explicitly, since the other ones are held fixed.

By the definition of the maximal box,
\begin{equation}
\norm{\psi}_{B(x,\tilde\delta)}^2
\geq
c_1 \norm{\psi}_{\L_1(k)}^2,  \text{ where }
\label{appendix-eq:c1}
c_1:=c_1(\tilde\delta,\beta):= \frac{C_{qUC}(\tilde\delta, \beta) }{(10\ceil{\sqrt{d}})^d}
\leq  1
\end{equation}
So we have obtained (\ref{appendix-aim}) for any point $x\in A(x_k)$ with $B(x,\tilde\delta)$ contained in $\Lambda_1(k)$.

\subsection{Second case}\label{appendix-second-est}

To cover the case of balls centered in points in $\L_{\frac{3}{10}}(x_k)\cap\L_1(k)$, we apply Corollary \ref{appendix-c:our_corollary}
two times in $\L_T(k)$ with two different sets $\Theta$.
The first step is a particular case of \ref{appendix-first-est}.

Consider the belt
\[
 \tilde A (x_k) = \L_1(k)\setminus \L_{\frac{6}{10}}(x_k)
\]
Notice that $\tilde A (x_k)  \neq \emptyset$ independent of the location of $x_k$ in the cube $\L_1(k)$ and it always
contains a cube of sidelength $1/10$, say $\L_{\frac{1}{10}}(y_k)$, for some $y_k\in \tilde A (x_k)$.
In particular,
\[
B(y_k,\frac{1}{20}) \subset     \L_{\frac{1}{10}}(y_k)     \subset \tilde A(x_k) \subset A(x_k)
\]
so taking $\tilde\delta=1/20$ in the previous estimate (\ref{appendix-eq:c1})  and recalling (\ref{appendix-eq:dombox})
\[
\norm{\psi}_{B(y_k,\frac{1}{20})}^2
\geq \frac{c_1(\frac{1}{20}, \beta)}{2T^d} \int_{\L_T(k)} |\psi|^2
\]
Therefore, we obtain the analog of (\ref{appendix-quotient}):
\begin{equation}
\frac{\norm{\psi}_{\L_T(k)}^2}{ \norm{\psi}_{B(y_k,\frac{1}{20})}^2}
\leq \frac{2T^d}{c_1(\frac{1}{20}, \beta)}:=\tilde\beta
\end{equation}
Note that $\tilde \beta\geq \beta $. Now we are prepared for applying Corollary \ref{appendix-c:our_corollary} a second time.
\smallskip

We have that for any $x\in \L_{\frac{3}{10}}(x_k)$, $\dist(x, \tilde A(x_k))\geq 3/20$.
In particular, for any $x\in\L_{\frac{3}{10}}(x_k)\cap\L_1(k)$  we  have that $\dist (x, B(y_k, \frac{1}{20}))\in [\frac{3}{20}, \sqrt{d})$.
Take $G=\L_T(k)$, $\Theta=B(y_k,\frac{1}{20})$ with $y_k$ as above, and apply Corollary \ref{appendix-c:our_corollary} again to obtain
\[
 \norm{\psi}_{B(x,\delta)}^2  \geq
C_{qUC}(\delta, \tilde\beta)
\norm{\psi}_{B(y_k,\frac{1}{20})}^2
\]
In particular, by taking $B(x,\tilde\delta):=B(y_k,\frac{1}{20})$ in (\ref{appendix-eq:c1}) we get that
\begin{equation}
\label{appendix-eq:second}
\norm{\psi}_{B(x,\delta)}^2
\geq
C_{lf}
\norm{\psi}_{\L_1(k)}^2
\quad \text{where }
C_{lf}:=
\frac{ C_{qUC}(\frac{1}{20}, \beta)  C_{qUC}(\delta, \tilde\beta) }{(10\ceil{\sqrt{d}})^d}
\end{equation}
Notice that  $C_{lf}\leq c_1$. Indeed,  we have
\begin{align*}
C_{lf}
&= \frac{ C_{qUC}(\frac{1}{20}, \beta)  C_{qUC}(\delta, \tilde\beta) }{(10\ceil{\sqrt{d}})^d}
 \leq
\frac{C_{qUC}(\delta, \tilde\beta) }{(10\ceil{\sqrt{d}})^d}
 \leq
\frac{C_{qUC}(\delta, \beta) }{(10\ceil{\sqrt{d}})^d}
=c_1 (\delta, \beta)
 \end{align*}
where we used that $\tilde \beta \geq \beta$ and 
the function $\beta \mapsto C_{qUC}(\delta, \beta) $ is monotone decreasing in the parameter $\beta$, 
cf. \eqref{appendix-eq:our_corollary} in Corollary \ref{appendix-c:our_corollary}.  
Therefore \eqref{appendix-eq:second} holds with the constant $C_{lf}$ in both possible cases, \ref{appendix-first-est} and \ref{appendix-second-est}. 
\end{proof}

\bigskip
%

%
%

\textbf{Acknowledgements}
C. R-M. is thankful for hospitality at TU Chemnitz, where
this work was done. C. R-M. was partially supported by ANR BLAN 0261.
I.V. would like to thank J.~Voigt for discussions related to extension properties of
Sobolev-space functions, and Ch.~Rose for help in the preparation of the manuscript.
The authors thank F.~Germinet and the anonymous referees for suggesting how to simplify the
proof of Theorem \ref{t:scale_free-UCP}.
%
\def\cprime{$'$}\def\polhk#1{\setbox0=\hbox{#1}{\ooalign{\hidewidth
  \lower1.5ex\hbox{`}\hidewidth\crcr\unhbox0}}}

\end{document}